\documentclass[12pt,letter]{article}
\usepackage[english]{babel}
\usepackage{amssymb}
\usepackage{amsthm}
\usepackage{amsfonts}
\usepackage{amsmath}
\usepackage{anysize}
\usepackage{bbm}
\usepackage{abstract}
\marginsize{2.4cm}{2.4cm}{0.85cm}{1.8cm}
\usepackage{xcolor}

\newtheorem{theorem}{Theorem}[section]
\newtheorem{lemma}[theorem]{Lemma}
\newtheorem{prop}[theorem]{Proposition}
\newtheorem{coro}[theorem]{Corollary}

\theoremstyle{definition}
\newtheorem{definition}[theorem]{Definition}

\newtheorem{remark}[theorem]{Remark}
\newtheorem*{remark*}{Remark}

 \theoremstyle{plain}
\newtheorem*{namedthm}{\namedthmname}
\newcounter{namedthm}

\makeatletter
\newenvironment{named}[1]
  {\def\namedthmname{#1}%
   \refstepcounter{namedthm}%
   \namedthm\def\@currentlabel{#1}}
  {\endnamedthm}

\newcommand{\CPP}{\mathbb{C}\mathbb{P}}

\newcommand{\Ec}{\mathcal{E}}

\newcommand{\setdef}{\ \vert \ }
\newcommand{\vep}{\varepsilon}
\newcommand{\psh}{{\rm PSH}}

\newcommand{\capacity}{{\rm Cap}}

\newcommand{\capphi}{{\rm Cap}_{\phi}}
\newcommand{\cappsi}{{\rm Cap}_{\psi}}
\newcommand{\capchi}{{\rm Cap}_{\chi}}

\newcommand{\Vol}{{\rm Vol}}

\newcommand{\ddbar}{\partial\bar\partial}

\newcommand{\AM}{{\rm I}}

\newcommand{\MA}{{\rm MA}}
\newcommand{\MV}{{\rm MV}}
\newcommand{\PSH}{{\rm PSH}}
\newcommand{\Amp}{{\rm Amp}}
\newcommand{\id}{\mathbbm{1}}

\newcommand{\Supp}{{\rm Supp}}

\usepackage{todonotes}

\usepackage{hyperref}
\hypersetup{
    bookmarks=true,         
    unicode=false,          
    pdftoolbar=true,       
    pdfmenubar=true,       
    pdffitwindow=false,    
    pdfstartview={FitH},   
    pdftitle={Log-concavity of volume and complex Monge-Amp\`ere equations with prescribed singularity},    
    pdfauthor={T. Darvas, E. Di Nezza, C.H. Lu},     
    colorlinks=true,       
   linkcolor=black,          
    citecolor=black,        
    filecolor=black,      
    urlcolor=black}           

\usepackage{mathtools}

\begin{document}

\title{Log-concavity of volume and complex Monge-Amp\`ere equations with prescribed singularity}

\author{Tam\'as Darvas, Eleonora Di Nezza, Chinh H. Lu}
\date{}

\maketitle

\begin{abstract} Let $(X,\omega)$ be a compact K\"ahler manifold. We prove the existence and uniqueness of solutions to complex Monge-Amp\`ere equations with prescribed singularity type. Compared to previous work, the assumption of small unbounded locus is dropped, and we work with general model type singularities. We state and prove our theorems in the context of big cohomology classes, however our results are new in the K\"ahler case as well.
As an application we confirm a conjecture by Boucksom-Eyssidieux-Guedj-Zeriahi concerning log-concavity of the volume of closed positive $(1,1)$-currents. Finally, we show that log-concavity of the volume in complex geometry corresponds to the Brunn-Minkowski inequality in convex geometry, pointing out a dictionary between our relative pluripotential theory and $P$-relative convex geometry. Applications related to stability and existence of csck metrics are treated elsewhere.
\end{abstract}

\section{Introduction}

Suppose $(X,\omega)$ is a compact connected K\"ahler manifold of complex dimension $n$. In this work 
we show that it is possible to solve complex Monge-Amp\`ere equations with prescribed singularity type, without any technical conditions. 

To put our results in historical context, we start with Yau's theorem \cite{Ya78}: given $f>0$ smooth with $\int_X f \omega^n = \int_X \omega^n$, it is possible to find a unique $u \in C^\infty(X,\mathbb{R})$ such that $\omega_u := \omega + i\ddbar u >0$ and
\begin{equation}\label{eq: Yau_eq}
\omega_u^n = f\omega^n \ \textup{ on } \ X.
\end{equation}
Geometrically, the above equation simply means that it is possible to prescribe the volume form of K\"ahler metrics within a K\"ahler class. 

Given additional geometric data, one is tempted to ask similar questions. To describe a motivating example, consider a finite number of complex submanifolds $D_j \subset X$. We ask: is it possible to find a solution to \eqref{eq: Yau_eq} on $X \setminus \cup_j D_j$, with the potential $u$  having prescribed asymptotics near the submanifolds $D_j$? Roughly speaking, when the asymptotics are governed by the $\log(\cdot)$ (or the $\log(-\log(\cdot))$) of the distances from the $D_j$, then the solution $u$ is said to have analytic singularity (or Poincar\'e type singularity) along the $D_j$ (see Section 2). Under various restrictive conditions, such problems were studied by Yau \cite[Section 9]{Ya78}, Tian-Yau \cite{TY87,TY90}, Phong-Sturm \cite{PS14}, Auvray \cite{Au17}, the two of us in \cite{DL17}, and many others. 

To deal with such questions collectively and in an efficient manner (allowing arbitrary asymptotics near $\cup_j D_j$) it is advantageous to consider a potential $\phi \in \textup{PSH}(X,\omega)$ that ``models" the singularity behavior near $\cup_j D_j$: we  simply ask that $u-\phi$ stays uniformly bounded on $X \setminus \cup_j D_j$, i.e. that $u$ and $\phi$ have the same \emph{singularity type}. 

This setup allows to disregard the potentially complicated geometry of the submanifolds and their intersections, and it also leads to a number of natural questions: is it possible to consider infinitely many divisors $D_j$? For what $\phi$ can we find a solution $u$, with the same singularity type as $\phi$? To what extent are such solutions unique? It turns out that all the information regarding well posedness of these problems is carried by the potential $\phi$, and the specific geometry of the $D_j$ can be ignored.

More concretely, in \ref{thm: main thm A}(i) below, we fully characterize the potentials $\phi$ for which a solution $u$ to \eqref{eq: Yau_eq} can be found, with the same singularity type as $\phi$. This theorem, along with its analog for Aubin-Yau type equations,  generalizes simultaneously the main result of Ko{\l}odziej \cite{Ko98} and the appropriate results Boucksom-Eyssidieux-Guedj-Zeriahi \cite{BEGZ10}. As applications, we fully resolve the log-concavity conjecture regarding volumes of positive currents from \cite{BEGZ10}, and we point out the close connection between our theorems and the Brunn-Minkowski theory of convex bodies. 

\medskip\noindent\textbf{Complex Monge-Amp\`ere equations with prescribed singularity.} With the above informal picture in mind, we lay down the precise details of our problem.
Suppose $\theta$ is a smooth $(1,1)$-form representing a big cohomology class on $X$. Given $u,v \in \textup{PSH}(X,\theta)$,  we say that 
\begin{itemize}\vspace{-0.1cm}
	\item $u$ is more singular than $v$, i.e., $u \preceq v$, if there exists $C\in \mathbb{R}$  such that $u\leq v+C$;\vspace{-0.1cm}
	\item $u$ has the same singularity as $v$, i.e., $u \simeq v$, if $u\preceq v$ and $v\preceq u$. \vspace{-0.1cm}
\end{itemize}
The classes $[u]$ of this latter equivalence relation are called \emph{singularity types}. 

Fixing $\phi \in \textup{PSH}(X,\theta)$ and $f \in L^p(X,\omega^n), \ f \geq 0, \ p >1$, we seek a solution to the following problem:
\begin{equation}\label{eq: CMAE_with_sing}
\begin{cases}
u \in \textup{PSH}(X,\theta),\\
\theta_u^n = f \omega^n,\\
[u]=[\phi],
\end{cases}
\end{equation}
where $\theta_u^n$ is understood in the sense of pluripotential theory, as the non-pluripolar Monge-Amp\`ere measure of $u$, introduced in \cite[Section 1.2]{BEGZ10}.  
 When $\theta$ is K\"ahler and $\phi =0$, \eqref{eq: CMAE_with_sing} reduces to Ko{\l}odziej's $L^{\infty}$-estimate \cite{Ko98} in the context of the Calabi-Yau theorem \cite{Ya78}.

By \cite[Theorem 1.2]{WN17} the correspondence $[u] \to \int_X \theta_u^n$ is well defined and monotone with respect to the (partial) ordering $\preceq$, and in \cite[Theorem 1.1]{DDL2} this was generalized to mixed non-pluripolar products. In particular, the normalization condition $\int_X \theta_\phi^n=\int_X f\omega^n>0$ becomes necessary in the above problem. 

As pointed out in \cite[Theorem 4.34]{DDL2}, it is only possible to solve the above equation for all $f \in L^p, \ p >1$ if we assume that $\phi$ is a potential with \emph{model type singularity}, that is $[\phi]= [P_\theta[\phi]]$ (i.e. $\phi - P_\theta[\phi]$ is bounded on $X$), 
where 
$$P_\theta[\phi] = \left (\sup\{\psi \in 
\textup{PSH}(X,\theta), \ \psi \leq 0 \textup{ and } \psi \preceq \phi \} \right)^*.$$
For an elaborate discussion on the relationship between the envelope $P_\theta$  and singularity types we refer to Section \ref{sect: preliminaries}. 

We now state our first main result, that provides unique solutions under these necessary conditions, not only to the  above problem, but also to a related one relevant to K\"ahler-Einstein geometry:
\begin{named}{Theorem A}\label{thm: main thm A}  Suppose that $[\phi]$ is a model type singularity. Let  $f \in L^p(\omega^n), p > 1$ be such that $f \geq 0$  and $\int_X f \omega^n=\int_X \theta_{\phi}^n > 0$. Then the following hold:\\
(i) There exists  $u\in \textup{PSH}(X,\theta)$, unique up to a constant, such that $[u]=[\phi]$ and 
\begin{equation}\label{eq: main_thm_eq1}
\theta_u^n=f \omega^n.
\end{equation}
(ii) For any $\lambda >0$  there exists a unique $v\in \textup{PSH}(X,\theta)$,  such that $[u]=[\phi]$ and 
\begin{equation}\label{eq: main_thm_eq2}
\theta_u^n=e^{\lambda u}f \omega^n.
\end{equation}
\end{named}

One of the main ingredients of this result is the relative Ko{\l}odziej estimate (Theorem \ref{thm: uniform estimate}). Recently, the two of us used this same result to approximate $L^1$ finite energy geodesic rays with $L^{\infty}$ geodesic rays, while assuring convergence of the radial K-energy (see \cite[Theorem 1.5]{DL18}), proving  (the uniform version of) Donaldson's geodesic stability conjecture for $L^\infty$ rays.

\begin{remark*} As mentioned earlier, by \cite[Theorem 4.34]{DDL2}, asking for $[\phi]$ to be a model type singularity is not only \emph{sufficient}, but also a \emph{necessary} (!) condition for the solvability of \eqref{eq: main_thm_eq1} for all $f \in L^p(X,\omega^n), \ p >1$. Consequently, model type singularities are truly natural, and appear in many different contexts of complex differential geometry, as described in \cite[Remark 1.6]{DDL2}.

Also, the assumption of non-vanshing mass $\int_X \theta_{\phi}^n > 0$ is important for well-posedness. Indeed, while in the case $\int_X \theta_{\phi}^n =\int_X f \omega^n = 0$ the potential $\phi$ trivially solves both \eqref{eq: main_thm_eq1} and \eqref{eq: main_thm_eq2}, this solution is not unique(!) in the singularity class $[\phi]$ (see Remark \ref{rem: non-unique_zero_mass}).
\end{remark*}

\begin{remark*} In Theorem \ref{thm: non pluripolar existence} and Theorem \ref{thm: AY} we actually show that $|u-\phi|$ is under control, in terms of only $p,\omega,\int_X \theta_\phi^n, \|f\|_p$, and $\lambda$, thus the above result generalizes the main result of Ko{\l}odziej \cite{Ko98}. Given that $\phi$ might have dense unbounded locus in $X$, the same is true for $u$, hence the regularity of $u$ can not be improved in this context, making our results optimal. 
\end{remark*}

 The above result extends \cite[Theorem 1.4]{DDL2}, where we assumed that $\phi$ has additionally small unbounded locus. 
In order to apply the variational techniques of \cite{BBGZ13} this technical condition was necessary. Here we take a completely different approach and we point out that generic model type singularities do not have small unbounded locus (see the example above \cite[Lemma 4.1]{DDL2}). 

As one of the novelties of the paper, we will construct solutions using  \emph{super-solution techniques}, and this will allow to overcome the difficulties with using integration by parts in the variational approach. In fact, our results will allow to obtain a version of  \ref{thm: main thm A} where $f\omega^n$ is replaced with a non-pluripolar measure $\mu$ satisfying the normalization condition $\int_X \theta_\phi^n = \int_X d\mu >0$. In this case however solutions will not have the same singularity type as $\phi$, they will come from the slightly bigger relative full mass class $\mathcal E(X,\theta,\phi)$ introduced in \cite{DDL2}.

\medskip\noindent\textbf{Log-concavity of the volume.} To give an application to \ref{thm: main thm A},  in our second main result we confirm the log-concavity conjecture of Boucksom-Eyssidieux-Guedj-Zeriahi. Let us recall some related terminology. Let $T_1,T_2,\ldots,T_n$ be closed positive $(1,1)$-currents on $X$. Naturally, there exist smooth closed $(1,1)$-forms $\theta^1,\ldots,\theta^n$ and potentials  $u_j \in \textup{PSH}(X,\theta^j)$ such that $T_j = \theta^j + i\ddbar u_j$. The product $\langle T_1 \wedge \ldots \wedge T_n \rangle$ is defined as follows:
$$\langle T_1 \wedge \ldots \wedge T_n \rangle := \theta^1_{u_1} \wedge \ldots \wedge \theta^n_{u_n}.$$

Related to the full mass of this product we establish the following result, conjectured in \cite[Conjecture 1.23]{BEGZ10}, informally referred to as the ``log-concavity conjecture'' of total masses:

\begin{named}{Theorem B} 	\label{thm: main thm B} 
	Let $T_1,...,T_n$ be closed positive $(1,1)$-currents on $X$. Then 
	\begin{equation}\label{Hodge type}
	\int_X \langle T_1 \wedge \cdots \wedge  T_n \rangle \geq \left (\int_X \langle T_1^n \rangle  \right )^{\frac{1}{n}} \cdots \left (\int_X \langle T_n^n \rangle  \right )^{\frac{1}{n}}. 
	\end{equation}
	In particular, $T\mapsto \left ( \int_X \langle T^n \rangle\right )^{\frac{1}{n}}$ is concave on the set of closed positive $(1,1)$ currents, and so is the map $T\mapsto \log \left ( \int_X \langle T^n \rangle\right )$. 
\end{named}

If equality holds in \eqref{Hodge type}, it does not necessarily mean that the singularity types of the $T_j$ are the same up to scaling (as one would perhaps expect). Still, it remains an interesting question to characterize the conditions under which equality is attained.

The correspondence $T \to \int_X \langle T^n \rangle$ vastly generalizes the process of associating volume to a line bundle $L \to X$ (see \cite[Section 1]{Bo02}), an essential concept in complex algebraic geometry (see \cite[Section 2.2]{La04}). From this point of view  \eqref{Hodge type} is a Hodge index-type inequality. For an introduction to Hodge index type inequalities in algebraic geometry, we refer to \cite[Section 1.6]{La04}.

In connection with the above theorem, a number of partial results have been obtained in the past. When $T_1,\ldots,T_n$ are smooth this result is due to Demailly \cite{De93}. When $X$ is projective it was proved in \cite[Corollary E]{BFJ09} that the map $\alpha\rightarrow  (\alpha^n)^{1/n}$ is strictly concave on the big and nef cone of the real N\'eron-Severi space $N_1(X)$.
As pointed out in \cite[Page 223]{BEGZ10}, in case the potentials of $T_1, \ldots, T_n$ have analytic singularity type, after passing to a log-resolution, the above result reduces to the nef version of an inequality of  Khovanski-Teissier (see \cite[Proposition 5.2]{De93}). 
 In addition to this, in \cite[Corollary 2.15]{BEGZ10} the above result  is proved in the special case when $\{T_1\}=\ldots =\{T_n\}$ and $T_1,\ldots,T_n$ have full mass. In \cite[Section 5.2]{DDL1} we generalized this to the case when $\{T_1\},\ldots,\{T_n\}$ are possibly different, but  $T_1,\ldots,T_n$ have full mass. In \cite[Theorem 1.8]{DDL2} we obtained the version of the conjecture when the potentials of $T_1,\ldots,T_n$ have small unbounded locus. Here we finally obtain the general form of the conjecture. What is more, following our method of proof, it is clear that generalizations of \ref{thm: main thm A} to k-Hessian type equations will pave the way to other types of Khovanskii-Teissier type inequalities (see \cite[Section 1.6]{La04}) in the context of big cohomology classes.

\medskip\noindent\textbf{Relation with convex geometry.} Using the tools developed in the present paper, in the presence of polycircular symmetry, it is possible to describe a \emph{dictionary} between $\phi$-relative pluripotential theory and $P$-relative convex geometry. This latter subject has been explored recently in \cite{BB13, Ba17, BBL18}, motivated by the study of K\"ahler-Ricci solitons, Bergman measures and Fekete points. 

As we point out, our analysis recovers many known results in convex geometry, while also strengthening the connection between the theory of the real and complex Monge-Amp\`ere measures:
\smallskip

$\bullet$ In the presence of polycircular symmetry, there is a one-to-one correspondence between model type singularities $[\phi]$ and convex bodies $P \subset \mathbb{R}^n$ (see Theorem \ref{thm: model_sing_convex_body}).

$\bullet$  In this context the log-concavity inequality (Theorem B) corresponds to the celebrated Brunn-Minkowski inequality and its variants for convex bodies (see Theorem \ref{thm: BM}).

$\bullet$ Theorem A and its generalization (Theorem \ref{thm: non pluripolar existence}) recovers a theorem of Berman--Berndtsson for the real  Monge-Amp\`ere equation (see Theorem \ref{resolution real MA}). Also, we positively answer a question of Berman--Berndtsson \cite[Remark 2.23]{BB13}, giving a precise result about the asymptotics of solutions to the real  Monge-Amp\`ere equation (see Remark \ref{rem:BB13 2.23}).

\smallskip
Moreover, our analysis suggests that the k-Hessian analog of \eqref{Hodge type} (alluded to at  the end of the previous paragraph) corresponds to the mixed volume inequalities of Alexandrov--Fenchel. Due to space constraints we don't explore such avenues further, but we are optimistic that many more results can be obtained via the parallel study of the complex and convex theories.

\medskip\noindent\textbf{Organization of the paper.} In Section \ref{sect: preliminaries}
we recall the terminology and results of \cite{DDL2} concerning relative pluripotential theory. In Section \ref{sect: relative MA capacity} we develop (relative) Monge-Amp\`ere capacity, giving a significant generalization of Ko{\l}odziej's $L^{\infty}$ estimate (see  Theorem \ref{thm: uniform estimate}). Using this last result, \ref{thm: main thm A} is proved in Section \ref{sect: MA equation} and Section \ref{sect: AY equation} (Theorems \ref{thm: non pluripolar existence} and \ref{thm: AY}). In Section \ref{sect: log concave} we settle the log-concavity conjecture, and in Section \ref{sect: real MA} we explore the connection with $P$-relative convex geometry.

\medskip\noindent\textbf{Acknowledgements.}  The first named author has been partially supported by BSF grant 2016173 and NSF grant DMS-1610202. The second and third named authors are partially supported by the French ANR project GRACK. We thank Hugues Auvray and L\'aszl\'o Lempert for useful discussions. We thank Norman Levenberg and Turgay Bayraktar for stimulating discussions related to $P$-pluripotential theory and convex bodies. 

\section{Preliminaries} \label{sect: preliminaries}

In this section we recall known results from (relative) finite energy pluripotential theory, developed in \cite{DDL1,DDL2} (especially \cite[Sections 1-3]{DDL2}), and establish some novel preliminary theorems.

\subsection{Non-pluripolar complex Monge-Amp\`ere measures and relative pluripotential theory}

Let $(X,\omega)$ be a compact K\"ahler manifold of dimension $n$ and fix $\theta$ a smooth closed $(1,1)$-form whose cohomology class is big. Our notation is taken from \cite{DDL2}.

A function $u: X \rightarrow \mathbb{R}\cup \{-\infty\}$ is called quasi-plurisubharmonic if locally $u= \rho + \varphi$, where $\rho$ is smooth and $\varphi$ is a plurisubharmonic function. We say that $u$ is $\theta$-plurisubharmonic ($\theta$-psh for short) if it is quasi-plurisubharmonic and $\theta_u:=\theta+i\ddbar u \geq 0$ in the weak sense of currents on $X$. We let $\PSH(X,\theta)$ denote the space of all $\theta$-psh functions on $X$. The class $\{\theta\}$ is {\it big} if there exists $\psi\in \psh(X,\theta)$ such that $\theta +i\ddbar \psi\geq \vep \omega$ for some $\vep>0$.  
 
A potential $u \in \textup{PSH}(X,\theta)$ has {\it analytic singularities} if  it can be written locally as  $u(z) = c \log \sum_{j=1}^k |f_j(z)|^2 + h(z),$ where $c>0$, the $f_j's$ are holomorphic functions and $h$ is  smooth. By the fundamental approximation theorem of Demailly \cite{Dem92}, if $\{\theta\}$ is big  there are plenty of $\theta$-psh functions with analytic singularities.  Following \cite{Bo04,BEGZ10} the ample locus of $\{\theta\}$ (denoted by $\Amp(\theta)$) is defined to be the set of all $x\in X$ such that there exists a $\theta$-psh function on $X$ with analytic singularities, smooth in a neighborhood of $x$. It follows from \cite[Theorem 3.17 (ii)]{Bo04}  that there exists a $\theta$-psh function $\psi$ with analytic singularities such that $\Amp(\theta)$ is the open set on which $\psi$ is smooth and $\psi=-\infty$ on $X\setminus \Amp(\theta)$.   

When $\theta$ is non-K\"ahler, elements of $\textup{PSH}(X,\theta)$ can be quite singular, and we distinguish the potential with the smallest singularity type in the following manner:
$$V_\theta := \sup \{u \in \textup{PSH}(X,\theta) \textup{ such that } u \leq 0\}.$$

A function $u\in \PSH(X,\theta)$ is said to have minimal singularities if it has the same singularity type as $V_{\theta}$, i.e., $[u]=[V_\theta]$.  By the analysis above it follows that $V_{\theta}$ is locally bounded in the Zariski open set $\Amp(\theta)$.

Given $\theta^1,...,\theta^n$ closed smooth $(1,1)$-forms representing big cohomology classes and $\varphi_j \in \textup{PSH}(X,\theta^j)$, $j=1,...n$, following the construction of Bedford-Taylor \cite{BT76,BT82,BT87} in the local setting, it has been shown in \cite{BEGZ10} that the sequence of positive measures
\begin{equation}\label{eq: k_approx_measure}
{\mathbbm 1}_{\bigcap_j\{\varphi_j>V_{\theta^j}-k\}}\theta^{1}_{\max(\varphi_1, V_{\theta^1}-k)}\wedge \ldots\wedge \theta^n_{\max(\varphi_n, V_{\theta^n}-k)}
\end{equation}
has total mass (uniformly) bounded from above and is non-decreasing in $k \in \Bbb R$, hence converges weakly to the so called \emph{non-pluripolar product} 
\[
\theta^1_{\varphi_1 } \wedge\ldots\wedge\theta^n_{\varphi_n }.
\]
The resulting positive measure does not charge pluripolar sets. In the particular case when $\varphi_1=\varphi_2=\ldots=\varphi_n=\varphi$ and $\theta^1=...=\theta^n=\theta$ we will call $\theta_{\varphi}^n$ the non-pluripolar measure of $\varphi$, which generalizes the usual notion of volume form in case $\theta_{\varphi}$ is a smooth K\"ahler form. As a consequence of Bedford-Taylor theory it can be seen that the measures in \eqref{eq: k_approx_measure} all have total mass less than $\int_X \theta_{V_\theta}^n$, in particular, after letting $k \to \infty$ we notice that $\int_X \theta_{\varphi}^n \leq \int_X \theta_{V_\theta}^n$. In fact it was recently proved in \cite[Theorem 1.2]{WN17} that for any $u,v \in \textup{PSH}(X,\theta)$  the following monotonocity property holds for the total masses:
$$v \preceq u \Longrightarrow \int_X \theta_v^n \leq \int_X \theta_u^n.$$
This result, together with the generalization \cite[Theorem 1.1]{DDL2}, opened the door to the development of relative finite energy pluripotential theory, as introduced in \cite[Sections 2-3]{DDL2}, whose terminology we now recall. 

\paragraph*{Relative finite energy class $\Ec(X,\theta,\phi)$.}
Fixing $\phi \in \textup{PSH}(X,\theta)$ one can consider only $\theta$-psh functions that are more singular than $\phi$. Such potentials form the set $\textup{PSH}(X,\theta,\phi)$. Since the map $[u] \to \int_X \theta_u^n$ is monotone increasing, but not strictly increasing, it is natural to consider the set of $\phi$-relative \emph{full mass potentials}:
$$\mathcal E(X,\theta,\phi) := \left\{u \in \textup{PSH}(X,\theta,\phi) \ \textup{ such that } \int_X \theta_u^n = \int_X \theta_\phi^n\right \}.$$
Naturally, when $v \in \textup{PSH}(X,\theta,\phi)$ we only have $\int_X \theta^n_v \leq \int_X \theta^n_\phi$. As pointed out in \cite{DDL2}, when studying the potential theory of the above space, the following well known envelope constructions will be of great help:
$$ P_\theta(\psi,\chi), \ P_\theta[\psi](\chi),  \ P_\theta[\psi] \in \textup{PSH}(X,\theta).$$ 
These were introduced by Ross and Witt Nystr\"om \cite{RWN14} in their construction of geodesic rays, using slightly different notation. Given $\psi,\chi \in \textup{PSH}(X,\theta)$, the starting point is the  ``rooftop envelope'' $P_\theta(\psi,\chi):=(\sup\{v \in \textup{PSH}(X,\theta), \ v \leq \min(\psi,\chi) \})^*$. This allows us to introduce
$$P_\theta[\psi](\chi) := \Big(\lim_{C \to +\infty}P_\theta(\psi+C,\chi)\Big)^*.$$

It is easy to see that $P_\theta[\psi](\chi)$ only depends on the singularity type of $\psi$. When $\chi = V_\theta$, we will simply write $P_\theta[\psi]:=P_\theta[\psi](V_\theta)$ and refer to this potential as the \emph{envelope of the singularity type} $[\psi]$.

Using such envelopes we conveniently characterized membership in $\mathcal E(X,\theta,\phi)$ in case $\phi = P[\phi]$ and $\int_X \theta_\phi^n >0$ (see \cite[Theorem 1.3]{DDL2}):

 \begin{theorem}\label{thm: DDL2_E_char} Suppose $\phi \in \textup{PSH}(X,\theta), \ \phi = P[\phi]$ and $\int_X \theta_\phi^n >0$. Then $u \in \mathcal E(X,\theta,\phi)$ if and only if $u \in \textup{PSH}(X,\theta,\phi)$ and $P_\theta[u]=\phi$.
\end{theorem}

\paragraph*{Model potentials.}
Potentials $\phi$ that satisfy $\phi = P[\phi]$ are called \emph{model potentials}, and play an important role in finite energy pluripotential theory, as evidenced in the statement of the above theorem. 
The connection with model type singularities $[u]$ (defined in the introduction) is as follows: in case $\int_X \theta_u^n >0$, it was proved in \cite[Theorem 3.12]{DDL2} that $P_\theta[P_\theta[u]]=P_\theta[u]$. To summarize, every model type singularity  with non-vanishing mass has a model potential representative. 

As further evidenced by the next lemma, potentials with model type singularity play a distinguished role in the theory: 

\begin{lemma}\label{lem: compactness and model type}
Let $\phi \in \PSH(X,\theta)$ with $\int_X \theta_{\phi}^n>0$. Then the following are equivalent: 
\begin{itemize}
\item[(i)] The set $\mathcal{F}:=\{u \in \PSH(X,\theta) \setdef \sup_X (u-\phi)=0\}$ is relatively compact in the $L^1$-topology of potentials.
\item[(ii)] $\phi$ has model type singularity.
\end{itemize}
\end{lemma}
\begin{proof}
Assume that $\phi$ has model type singularity and let $C_0>0$ be a constant such that $-C_0 + P_{\theta}[\phi] \leq \phi \leq P_{\theta}[\phi]+C_0$ on $X$. Then 
$$
\sup_X (u-P_{\theta}[\phi]) -C_0 \leq \sup_X (u-\phi) \leq \sup_X (u-P_{\theta}[\phi]) +C_0. 
$$
Now, observe that any $u \in \PSH(X,\theta)$ which is more singular than $P_\theta[\phi]$ satisfies $u-\sup_X u\leq P_\theta[\phi]\leq 0$, hence $\sup_X u=\sup_X (u-P_{\theta}[\phi])$. Therefore, $\mathcal F$ is contained in the following set 
$$
\{u\in \PSH(X,\theta) \ : \ -C_0 \leq \sup_X u \leq C_0\}.
$$
The latter set is compact in the $L^1$-topology as follows from \cite[Proposition 2.6]{GZ05}.

Next we prove that ``not $(ii)$" implies ``not $(i)$". Assume that $\phi$ does not have model type singularity, i.e. $\phi-P_{\theta}[\phi]$ is unbounded.  Consider $u_t:= P_\theta(\phi+t,P_\theta[\phi])$, $t>0$. Then $u_t \leq P_\theta[\phi]\leq 0$ and also $u_t \leq \phi+t$ for all $t$. 

We claim that $\sup_X (u_t-t-\phi) =0.$ We are going to argue this by contradiction. If it is not the case then, by \cite[Lemma 3.7]{DDL2} the (non-pluripolar) Monge-Amp\`ere measure of $u_t$ is concentrated on the set $\{P_\theta(\phi+t,P_\theta[\phi])=P_\theta[\phi]\}$, hence 
$$
\int_{\{u_t <P_\theta[\phi]\}} \theta_{u_t}^n  =0. 
$$
Since $\int_X \theta_{u_t}^n =\int_X \theta_{\phi}^n=\int_X \theta_{P_\theta[\phi]}^n>0$, the domination principle (\cite[Proposition 3.11]{DDL2}) then ensures that $u_t=P_\theta[\phi]$. On the other hand, $u_t\leq \phi +t$, hence $P_\theta[\phi]\leq \phi +t$. Since $\phi$ is more singular than $P_{\theta}[\phi]$ we infer that $\phi-P_{\theta}[\phi]$ is bounded  which is a contradiction. This proves the  claim.

It then follows that $u_t-t\in \mathcal{F}$. However $u_t-t\leq P_\theta[\phi]-t \searrow -\infty$ as $t\rightarrow \infty$. This implies that the set $\mathcal{F}$ is not relatively compact, as desired.
\end{proof}

Next we point out a slight generalization of the comparison principle of \cite{DDL2}, that will be used in the sequel:

\begin{lemma}
\label{lem: DDL2 comparison principle}
Assume that $u,v\in \PSH(X,\theta)$ and $P[u]$ is less singular than $v$. Then 
$$\int_{\{u<v\}} \theta_v^n \leq \int_{\{u<v\}} \theta_u^n.$$
\end{lemma}
\begin{proof}
We can assume that $u,v \leq 0$. Let $\varphi = \max(u,v)$. Then $u,\varphi \in \Ec(X,\theta, P[u])$. Indeed, Theorem \cite[Theorem 2.3]{DDL2} gives that $\int_X \theta_u^n = \int_X \theta_{P[u]}^n$. Also, since $u \leq \varphi \leq P[u]$, \cite[Theorem 1.2]{WN17} gives that $\int_X \theta_\varphi^n = \int_X \theta_u^n=\int_X \theta_{P[u]}^n$.

The comparison principle in \cite[Corollary 3.6]{DDL2} and the locality of the complex Monge-Amp\`ere measure with respect to the plurifine topology gives the result: 
$$\int_{\{u<v\}} \theta_v^n = \int_{\{u<\varphi\}} \theta_{\varphi}^n \leq \int_{\{u<\varphi\}} \theta_u^n = \int_{\{u<v\}} \theta_u^n. $$
\end{proof}

For additional technical results regarding the potential theory of $\mathcal E(X,\theta,\phi)$, we refer to \cite[Section 3]{DDL2}. 

\subsection{The relative finite energy class $\mathcal E^1(X,\theta,\phi)$}

Under the assumption of small unbounded locus, the finite energy class $\mathcal E^1(X,\theta,\phi)$ was introduced in \cite{DDL2} with the goal of developing a variational approach to \eqref{eq: CMAE_with_sing}, generalizing the results of \cite{BBGZ13}. Though we take a different angle on equations with prescribed singularity type in this work, this space will still play an important role in the sequel. We start with the definition:
$$\mathcal E^1(X,\theta,\phi) = \left\{u \in \mathcal E(X,\theta,\phi) \ \textup{ such that } \int_X |u -\phi| \theta_u^n < + \infty \right\}.$$
Let us note that, in the case of $\phi$ having small unbounded locus, the above definition of $\mathcal{E}^1$ is equivalent to the one given in \cite[page 13]{DDL2} using the Monge-Amp\`ere energy $\AM_\phi$.  In the case of a general $\phi$ (i.e. not necessarily with small unbounded locus), the above definition is more convenient since in this setting the definition of the energy $\AM_\phi$ is quite delicate.

In our first result we generalize the \emph{fundamental inequality}  \cite[Lemma 2.3]{GZ07} from the K\"ahler case to our context:

\begin{lemma}\label{lem: GZ fundamental inequality}
	Let $\phi$ be a model potential with $\int_X \theta_{\phi}^n>0$. Assume that $u,v \in \Ec(X,\theta,\phi)$ are such that $v\leq u\leq 0$. Then  
	\[
	\int_X |u-\phi| \theta_u^n  \leq 2^{n+1} \int_X |v-\phi| \theta_v^n. 
	\]
In particular, if $v \in \mathcal E^1(X,\theta,\phi)$ then $u\in \mathcal E^1(X,\theta,\phi)$.
\end{lemma}
Note that above we don't rule out the possibility that the quantities in the above inequality might be infinite.
\begin{proof}
We first point out that we actually have $v\leq u\leq \phi \leq 0$. Indeed, since $u,v \in \mathcal E(X,\theta,\phi)$ we get that $v \leq u \leq P[\phi]=\phi \leq 0$.  

We also recall that for a Borel measure $\mu$  and a positive measurable function $f$ on $X$ we have 
	\[
	\int_X f\, d\mu = \int_0^{+\infty} \mu(f>t) dt. 
	\]  
	Applying this to $f=|u-\phi|=\phi-u$ and $\mu= \theta_u^n$ we obtain 
	\[
	\int_X |u-\phi| \theta_u^n =\int_{0}^{+\infty} \theta_u^n(u<\phi-t) dt = 2\int_0^{+\infty} \theta_u^n (u<\phi-2t) dt. 
	\]
	Observe that, since $\phi\geq u\geq v$ the following inclusions of sets hold
	\[
	\{u<\phi-2t\} \subset \{v< (u+ \phi)/2 -t\} \subset \{v < \phi-t\}. 
	\]
	The comparison principle \cite[Corollary 3.6]{DDL2} and the fact that $\theta_u^n\leq 2^n \theta_{\frac{u+\phi}{2}}^n$ then yield
	\begin{flalign*}
		\int_X |u-\phi| \theta_u^n & =  2\int_0^{+\infty} \theta_u^n (u<\phi-2t) dt \leq  2\int_0^{+\infty} \theta_u^n (v<(\phi+u)/2-t) dt \\
		& \leq 2^{n+1}  \int_0^{+\infty} \theta_{\frac{u+\phi}{2}}^n (v<(u+\phi)/2-t) dt \\
		&\leq 2^{n+1}  \int_0^{+\infty} \theta_v^n (v<(\phi+u)/2-t) dt\\
		& \leq 2^{n+1}  \int_0^{+\infty} \theta_v^n (v<\phi-t) dt = 2^{n+1} \int_X |v-\phi| \theta_v^n. 
	\end{flalign*}
\end{proof}

Next we generalize another result from \cite{GZ07}:

\begin{lemma}\label{lem: GZ fundamental inequality 2} 
	Let $\phi$ be a model potential with $\int_X \theta_{\phi}^n>0$. Suppose $u,v \in \mathcal E(X,\theta,\phi)$ and $u,v \leq 0$. Then the following hold:
\begin{equation*}
\int_X |u - \phi| \theta_v^n \leq 2\int_X |u-\phi| \theta_u^n + 2\int_X |v-\phi| \theta_v^n.
\end{equation*}
\end{lemma}
\begin{proof}
As in the previous lemma, we actually have $u\leq \phi \leq 0$ and $v \leq \phi \leq 0$ so we can start writing
$$\int_X |u - \phi| \theta_v^n = 2 \int_0^{+\infty} \theta_v^n(u - \phi \leq - 2t ) dt.$$
To continue we notice that 
\begin{flalign*}
\{ u  -\phi \leq - 2t\}\subset 
 \{v - \phi \leq - t \} \cup \{u \leq v-t \}.
\end{flalign*}
Putting the above together, and using the comparison principle \cite[Corollary 3.6]{DDL2}, we can continue to finish the proof:
\begin{flalign*}
\int_X |u - \phi| \theta_v^n &\leq 2\int_0^{+\infty} \theta_v^n(v - \phi \leq - t) + 2\int_0^{+\infty} \theta_v^n(u \leq v-t ) \\
&\leq  2 \int_X |v -\phi| \theta_v^n + 2  \int_0^{+\infty} \theta_u^n(u \leq v-t ) \\
&\leq  2 \int_X |v -\phi| \theta_v^n + 2  \int_0^{+\infty} \theta_u^n(u \leq \phi-t )\\
&= 2 \int_X |v -\phi| \theta_v^n + 2  \int_X |u -\phi| \theta_u^n.
\end{flalign*}
\end{proof}
Next we point out that $\mathcal E^1(X,\theta,\phi)$ is $L^1$-stable in a certain sense:
\begin{lemma}
	\label{lem: BEGZ phi}

	Let $\phi$ be a model potential with $\int_X \theta_{\phi}^n>0$. Assume that the sequence $u_j \in \mathcal E^1(X,\theta,\phi)$ is normalized by $\sup_X u_j=0$, with each member satisfying
	\[
	\int_X |u_j-\phi|\theta_{u_j}^n \leq A,
	\]
	for some $A>0$. If $u_j \to u\in \PSH(X,\theta)$ in $L^1(X,\omega^n)$, then $u\in \Ec(X,\theta,\phi)$ and 
\begin{equation*}
	\int_X |u-\phi| \theta_u^n\leq 2^{n+3} A.
\end{equation*}
\end{lemma}

\begin{proof}
First let us assume that $u_j \searrow u$. By Lemma \ref{lem: GZ fundamental inequality 2} we have that 
$$
\int_X |u_j - \phi| \theta_{u_k}^n \leq 2\int_X |u_j-\phi| \theta_{u_j}^n + 2\int_X |u_k - \phi| \theta_{u_k}^n \leq 4 A.
$$
Fixing $C > 0$, since $u_j \leq \max(u_j,\phi-C) \leq \phi$, we arrive at 
$$
\int_X |\max(u_j,\phi-C) - \phi| \theta_{u_k}^n \leq \int_X |u_j - \phi| \theta_{u_k}^n \leq 4 A.
$$
Since $|\max(u_j,\phi-C) - \phi|$ is uniformly bounded and quasi-continuous, we can apply \cite[Theorem 2.3]{DDL2} to conclude that
$$
\int_X |\max(u_j,\phi-C) - \phi| \theta_{u}^n \leq  \liminf_k\int_X |\max(u_j,\phi-C) - \phi| \theta_{u_k}^n  \leq 4 A.
$$
Moreover, we notice that $u_j^C \searrow u^C$, where $u^C_j := \max(u_j,\phi-C)$, and $u^C := \max(u,\phi-C)$, and that $u^C \searrow u$. Letting $j \to +\infty$ and then $C \to \infty$, the monotone convergence theorem implies that $\int_X |u - \phi| \theta_{u}^n\leq 4 A$.

In the general case, when $u_j \to u$ in $L^1$, consider the sequence $v_j := \big(\sup_{k \geq j} u_j\big)^*  \geq u_j$. It is clear that $v_j \searrow u$, hence by Lemma \ref{lem: GZ fundamental inequality} we can conclude that
\begin{equation*}
\sup_j \int_X |v_j - \phi| \theta_{v_j}^n \leq 2^{n+1} \sup_j \int_X |u_j - \phi| \theta_{u_j}^n \leq 2^{n+1} A. 
\end{equation*}

To address that $u \in \mathcal E(X,\theta,\phi)$, we notice that, for $C>0$ fixed, $\phi \geq v^C_j\geq u^C \in \Ec(X,\theta,\phi)$ and $v_j^C$ decreases to $u^C$. Hence, since $\phi\geq v_j^C\geq u_j \in \Ec(X,\theta,\phi)$ we can use the first step and Lemma \ref{lem: GZ fundamental inequality} to conclude that 
\begin{equation}
\label{eq: fund inequality GZ07}
\int_X |u^C - \phi| \theta_{u^C}^n \leq 4 \sup_j\int_X |v_j^C - \phi| \theta_{v_j^C}^n\leq 2^{n+3} \sup_j\int_X |u_j - \phi| \theta_{u_j}^n \leq 2^{n+3}A.
\end{equation}
In particular, this implies that $\int_{\{u \leq \phi-C\}} \theta_{u^C}^n\leq \frac{1}{C}\int_X |u^C-\phi|\theta_{u^C}^n \leq \frac{2^{n+3}A}{C}$. It then follows from \cite[Lemma 3.4]{DDL2} that $u \in \mathcal E(X,\theta,\phi)$. Finally from \eqref{eq: fund inequality GZ07} and the plurifine property we have
$$
\int_{\{u>\phi-C\}} |u^C-\phi|\theta_u^n \leq 2^{n+3}A.
$$
Now, letting $C\to +\infty$ and using the monotone convergence theorem we finish the proof. 
\end{proof}

Finally, we prove an estimate that will be useful in showing that certain equations with prescribed singularity have solutions:

\begin{lemma}\label{lem: energy capacity estimate}
	Let $\phi$ be a model potential with $\int_X \theta_{\phi}^n>0$. Let $u \in \Ec(X,\theta,\phi)$ be such that $\sup_X u=0$, and  let $\mu$ be a positive Borel measure such that $\mu \leq B \capphi$ for some $B>0$. Then
	\[
	\int_X |u-\phi|^2d\mu \leq C  \left(\int_X |u-\phi|\theta_u^n +1\right), 
	\]
	where $C>0$ only depends on $B$, $\theta$ and $\omega$.
\end{lemma}
For the definition of the relative Monge-Amp\`ere capacity $\capphi$ we refer to \eqref{eq: capphi_def}. The proof builds on the arguments of \cite[Lemma 4.18]{DDL2}. 
\begin{proof}
We first express the left-hand side in the following manner:
\begin{flalign*}
\int_X |u-\phi|^2 d\mu& = 2\int_0^{+\infty} t\mu(u<\phi-t)dt = 4\int_0^{+\infty} t\mu(u<\phi-2t)dt\\
& \leq 4 B\int_X \theta_\phi^n + 4B \int_1^{+\infty}t \capphi(u<\phi-2t)dt. 	
\end{flalign*}
		Next, we use the comparison principle to estimate $\capphi(u<\phi-2t), t>1$. It suffices to prove that 
    \begin{equation*}
    	 \int_{1}^{+\infty} t\capphi(u<\phi-2t) dt \leq C\left (\int_X |u-\phi|\theta_u^n +1\right),
    \end{equation*}
    for some uniform constant $C:=C(X,n,\theta,\omega)>0$.   Fix $v\in \PSH(X,\theta)$ such that $\phi-1\leq v \leq \phi$. 
        For each $t>1$ we set $u_t:=t^{-1}u+(1-t^{-1})\phi$.  Observe that the following inclusions hold
	\[
	(u<\phi-2t)=(t^{-1}u + \phi - t^{-1}\phi<\phi-2) \subset (u_t<v-1) \subset (u_t<\phi-1)= (u<\phi-t). 
	\]
 It thus follows from the comparison principle \cite[Corollary 3.6]{DDL2} that 
	\begin{equation}\label{ineq measures}
	\theta_{v}^n(u<\phi-2t) \leq \theta_{v}^n(u_t<v-1) \leq \theta_{u_t}^n (u_t<v-1)\leq \theta_{u_t}^n (u<\phi-t).
	\end{equation}
	Expanding $\theta_{u_t}^n$ we see that 
	\begin{equation}\label{ineq u}
	\theta_{u_t}^n \leq Ct^{-1}\sum_{k=1}^n \theta_{u}^k\wedge \theta_{\phi}^{n-k} +  \theta_{\phi}^n, \ \ \forall t>1, 
	\end{equation}
for a uniform constant $C=C(n)>0$. Since $\theta_{\phi}^n$ has bounded density with respect to Lebesgue measure (see \cite[ Theorem 3.8]{DDL2}), using \cite[Theorem 2.50]{GZ17} we infer that 
\begin{equation}
	\label{eq: GZ 05 CLN} 
	\theta_{\phi}^n (u<\phi-t) \leq A \int_{\{u\leq - t\}} \omega^n \leq    A e^{-at} \int_X  e^{-au}\omega^n\leq A'e^{-at}  , 
\end{equation}
for some uniform constants $a,A, A'>0$ depending only on $n,\theta,\omega,X$. 
  Combining this with \eqref{ineq measures} and \eqref{ineq u} and taking the supremum over all candidates $v$ for the capacity $\capphi$ we get that 
	\begin{eqnarray*}
		\int_1^{\infty} t\capphi(u<\phi-2t)dt &\leq &  \int_1^{\infty} t\theta_{u_t}^n(u<\phi-t)dt\\
        &\leq &  C \int_1^{\infty}  \sum_{k=1}^n \theta_{u}^k\wedge \theta_{\phi}^{n-k}(u<\phi-t)dt + \int_1^{\infty}  t\theta_{\phi}^n(u<\phi-t) dt.
	\end{eqnarray*}
By \eqref{eq: GZ 05 CLN} we have  $\int_1^{\infty}  t\theta_{\phi}^n(u<\phi-t) dt \leq A'\int_1^{\infty}  t e^{-at} dt<+\infty $.
	Using the partial comparison principle \cite[Corollary 3.16]{DDL2} we get 
	\[
	\theta_u^k \wedge \theta_{\phi}^{n-k} (u<\phi-t) \leq \theta_u^n (u<\phi-t), \ \forall k \in \{1,...,n\}.  
	\]
	Combing the last two estimates we finally get the result.
\end{proof}

\subsection{Stability of subsolutions and supersolutions}

Let us consider momentarily the equation $\theta_u^n = \mu, \ \ u \in \textup{PSH}(X,\theta),$
where $\mu$ is a positive non-pluripolar Borel measure. Informally speaking, we say that $v \in \textup{PSH}(X,\theta)$ is a \emph{subsolution} to this equation if $\theta_v^n \geq \mu$. Analogously, we say that $u$ is a \emph{supersolution} if $\theta_u^n \leq \mu$.  In this short subsection we point out that subsolutions/supersolutions are stable under taking certain natural operations. 

It is well known that subsolutions are preserved under taking maximums (in our context see \cite[Lemma 4.27]{DDL2}). In addition to this, the $L^1$-limit of subsolutions is also a subsolution:
\begin{lemma}\label{lem: stability of subsolutions}
	Let $(u_j)$ be a sequence of $\theta$-psh functions such that $\theta_{u_j}^n \geq f_j \mu$, where $f_j \in L^1(X,\mu)$ and $\mu$ is a positive non-pluripolar Borel measure on $X$. Assume that $f_j$ converge in $L^1(X,\mu)$ to $f \in L^1(X,\mu)$,  and $u_j$ converge in $L^1(X,\omega^n)$ to $u\in \PSH(X,\theta)$. Then  $\theta_u^n \geq f \mu$. 
\end{lemma}
\begin{proof}
	By extracting a subsequence if necessary, we can assume that $f_j$ converge $\mu$-a.e. to $f$. For each $k$ we set $v_k:= (\sup_{j\geq k} u_j)^*$. Then $v_k$ decreases pointwise to $u$ and \cite[Lemma 4.27]{DDL2} gives
	\[
	\theta_{v_k}^n \geq \left(\inf_{j\geq k} f_j \right) \mu. 
	\]
To explain our notation below, for $t>0$ and a function $g$ we set $g^t:= \max(g,V_{\theta}-t)$. 

Note that $\{ u >V_\theta-t\}\subset \{ v_k >V_\theta-t\}$. Multiplying both sides of the above estimate with $\id_{\{u >V_{\theta}-t\}}$, $t>0$ and using the locality of the complex Monge-Amp\`ere operator with respect to the plurifine topology we arrive at 
	\[
	\theta_{v_k^t}^n \geq \id_{\{u >V_{\theta}-t\}} \left(\inf_{j\geq k} f_j\right) \mu.
	\]
	Note that for $t>0$ fixed $v_k^t$ decreases to $u^t$ all having minimal singularity type. Letting $k\to +\infty$ and using \cite[Theorem 2.17]{BEGZ10}, we obtain 
	\[
	\theta_{u^t}^n \geq \id_{\{u >V_{\theta}-t\}} f \mu,\ t>0.
	\]
  	Again, multiplying both sides with $\id_{\{u >V_{\theta}-t\}}$, $t>0$, and using the locality of the complex Monge-Amp\`ere operator with respect to the plurifine topology we arrive at 
	\[
	\id_{\{u >V_{\theta}-t\}}\theta_{u}^n \geq \id_{\{u >V_{\theta}-t\}} f \mu.
	\]
	Finally, letting $t\to +\infty$ we obtain the result. 
\end{proof}

The minimum of two $\theta$-psh potentials is not $\theta$-psh anymore, but the $P_\theta(\cdot,\cdot)$ operator replaces effectively the role of the pointwise minimum, and we have the following result regarding stability of ``minimums" of supersolutions: 

\begin{lemma}\label{lem: MA of rooftop envelope}
Suppose that $u,v \in \textup{PSH}(X,\theta)$ and $P_\theta(u,v) \in \PSH(X,\theta)$ are such that $\theta_u^n \leq \mu$ and $\theta_v^n \leq \mu$ for some  Borel measure $\mu$. 
Then $\theta_{P_\theta(u,v)}^n\leq \mu$. 
\end{lemma}

\begin{proof}
By replacing $\mu$ with $\id_{X\setminus P}\mu$, where $P:= \{u=v=-\infty\}$, we can assume that $\mu(P)=0$. Since $\mu(X)<+\infty$, the function $r \to  \mu(\{u \leq v+r \})$ is monotone increasing. Such functions have at most a countable number of discontinuities, hence for almost every $r \geq 0$ we have that $\mu(\{u=v+r\})=0$. 
For such $r$ we set $\varphi_r:=P_{\theta}(\min(u,v+r))$, and note that $\varphi_r \searrow P_\theta(u,v)$ as $r\rightarrow 0$.  
It then follows from \cite[Lemma 3.7]{DDL2} that we can write 
	\[
	\theta_{\varphi_r}^n \leq \mathbbm{1}_{\{\varphi_r=u\}} \theta_u^n+\mathbbm{1}_{\{\varphi_r=v+r\}}\theta_v^n \leq \left( \mathbbm{1}_{\{\varphi_r=u\}} +\mathbbm{1}_{\{\varphi_r=v+r\}} \right) \mu \leq \mu,
	\]
 where in the last inequality we used the fact that $\mu(\{u=v+r\})=0$. Letting $r\searrow 0$, we use \cite[Theorem 2.3]{DDL2} to arrive at the conclusion. 
\end{proof}

\section{The relative Monge-Amp\`ere capacity} \label{sect: relative MA capacity}

We recall the circle of ideas related to the $\phi$-relative Monge-Amp\`ere capacity. This notion has its roots in \cite{DL15,DL17}, and it was treated in detail in \cite{DDL2} under the assumption of small unbounded locus on $\phi$.  

The main result of this section is Theorem \ref{thm: uniform estimate}, which is a significant generalization of Ko{\l}odziej's $L^\infty$ estimate \cite{Ko98} to our relative context, that will help not only with the regularity of the solutions to our equations, but also with showing the solutions exist to begin with. 

We start by introducing the main concepts.
For this we fix $\chi \in \PSH(X,\theta)$.
\begin{definition}
 Let $E$ be a Borel subset of $X$. We define the $\chi$-relative capacity of $E$ as 
\begin{equation}\label{eq: capphi_def}
 \textup{Cap}_\chi(E) := \sup \left \{\int_E \theta_{u}^n \setdef u \in \PSH(X,\theta),
  \ \chi-1\leq u \leq \chi \right \}. 
\end{equation}
Exactly the same proof as \cite[Lemma 4.2]{DDL2} shows that $\textup{Cap}_\chi$ is inner regular, i.e., 
\[
\textup{Cap}_\chi(E) =\sup \{\textup{Cap}_\chi(K)\setdef K\subset E\ ; \ K \ \textrm{is compact}\}.
\]
Moreover it is elementary to see that $\textup{Cap}_\chi$ is continuous along increasing sequences, i.e., if $\{E_j\}_j$ increases to $E$ then 
\[
\textup{Cap}_\chi (\cup E_j) = \lim_{j} \textup{Cap}_\chi(E_j). 
\]
In particular, if $\psi$ is a quasi-psh function then the function $t\mapsto \textup{Cap}_\chi(\psi <\chi-t)$ is right-continuous on $\mathbb{R}$. This is an important  ingredient in proving analogs of  Ko{\l}odziej's $L^{\infty}$ estimate in this context (see Theorem \ref{thm: uniform estimate} below). \\
The \emph{relative} $\chi$-\emph{extremal function} of $E$ is defined as 
 \[
 h_{E,\chi}:= \sup \{u \in \PSH(X,\theta) \setdef 
 u \leq \chi-1\ \textrm{on} \ E\ \textrm{and} \ u \leq \chi \ \textrm{on}\ X\}.
 \]
 The \emph{global} $\chi$-\emph{extremal function} of $E$ is defined as 
 \[
 V_{E,\chi} := \sup \{ u \in \textup{PSH}(X,\theta,\chi) \setdef u\leq \chi \ \textrm{on} \ E 
 \}. 
 \]
\end{definition}
We set $M_\chi(E):=\sup_X V_{E,\chi}^*$, where $V_{E,\chi}^*$ denotes the upper semicontinuous regularization of $V_{E,\chi}$. 
The Alexander-Taylor capacity is then defined as $T_\chi (E):= \exp(-M_\chi(E))$.

\vspace{2mm}

By a word for word adaptation of the proof of \cite[Lemma 4.3]{DDL2} we obtain that sets with zero capacity are small:

\begin{lemma}\label{lem: zero_cap} Let $B \subset X$ be a Borel set. Then $\textup{Cap}_\chi(B)=0$ if and only if $B$ is pluripolar.
\end{lemma}

In similar spirit, we mention that $M_\chi(B)=+\infty$ implies that $\textup{Cap}_\chi(B)=0$. Indeed this is a consequence of \cite[Lemma 4.8]{DDL2}. 

\subsection{A relative version of Ko{\l}odziej's estimate}

In the next theorem, we give a significant generalization of Kolodziej's $L^\infty$ estimate. Though the main line of the proof is similar to the one in \cite{BEGZ10}, the statement will be flexible enough to help us with proving both the existence and regularity of solutions to equations with prescribed singularity. 

\begin{theorem}\label{thm: uniform estimate} 
	Fix $a\in [0,1),A>0$, $\chi \in \PSH(X,\theta)$ and  $0\leq f \in L^p(X,\omega^n)$ for some $p>1$. Assume that  $u\in \PSH(X,\theta)$, normalized by $\sup_X u=0$, satisfies 
	\begin{equation}
		\label{eq: volume cap domination 0}
		\theta_u^n \leq f\omega^n + a\theta_{\chi}^n.
	\end{equation}
	Assume also that  
	\begin{equation}
		\label{eq: volume cap domination}
		\int_E f\omega^n \leq A [\capacity_{\chi}(E)]^2,
	\end{equation}
	for every Borel subset $E\subset X$. If $P[u]$ is less singular than $\chi$ then 
$$\chi -\sup_X \chi- C\Big(\|f\|_{L^p},p,(1-a)^{-1},A\Big) \leq u.$$
\end{theorem}
\begin{proof}
By adding a constant to $\chi$ we can assume that $\sup_X \chi=0$. 
For $t>0$ we set 
$$
g(t):= [\capchi(u<\chi-t)]^{1/n}.
$$ 
Let $s\in [0,1]$ and suppose $v\in \PSH(X,\theta)$ satisfies $\chi-1\leq v\leq \chi$. Since $P[u]$ is less singular than $\chi$, the comparison principle (Lemma \ref{lem: DDL2 comparison principle})  gives 
\begin{eqnarray*} 
	s^n\int_{\{u<\chi-t-s\}} \theta_v^n & \leq  & s^n \int_{\{u<(1-s)\chi + sv -t\}} \theta_v^n \leq \int_{\{u<(1-s)\chi + sv -t\}} \theta_{(1-s)\chi +sv}^n \\
	&\leq & \int_{\{u<(1-s)\chi + sv -t\}} \theta_{u}^n \leq  \int_{\{u<\chi-t\}} \theta_u^n,
\end{eqnarray*} 
hence taking supremum over all candidates $v$ we arrive at 
\begin{equation}\label{eq: cap estimate 1}
	s^n\capchi(u<\chi-t-s)  \leq \int_{\{u<\chi-t\}} \theta_u^n. 
\end{equation}

For each $t>0$, since $P[u]$ is less singular than $\chi$, the comparison principle (Lemma \ref{lem: DDL2 comparison principle}) and the assumption \eqref{eq: volume cap domination 0} give
\begin{eqnarray*}
\int_{\{u<\chi-t\}}  \theta_u^n  & \leq & \int_{\{u<\chi-t\}} f\omega^n + a \int_{\{u<\chi-t\}} \theta_{\chi}^n \leq \int_{\{u<\chi-t\}} f\omega^n + a \int_{\{u<\chi-t\}} \theta_{u}^n.
\end{eqnarray*}
Since $a\in [0,1)$ we thus get $$\int_{\{u<\chi-t\}}  \theta_u^n \leq \frac{1}{1-a}  \int_{\{u<\chi-t\}} f\omega^n.$$
Combining this with \eqref{eq: cap estimate 1} we get
\begin{equation}\label{eq: cap estimate 3}
	s^n \capchi(u<\chi-t-s)  \leq \frac{1}{1-a}\int_{\{u<\chi-t\}} f\omega^n.
\end{equation} 
Therefore, combining  \eqref{eq: volume cap domination} with \eqref{eq: cap estimate 3}  we obtain
\begin{equation*}
	s^n\capchi (u<\chi-t-s)    \leq \frac{A}{1-a}  [ \capchi(u<\chi-t)]^2,
\end{equation*}
which implies
\[
sg(t+s) \leq B g^2(t), \ \forall t > 0, \forall s \in [0,1], 
\]
where $B=(A/(1-a))^{1/n}$. As we have already pointed out in the beginning of this section, $g :\mathbb{R}^+ \rightarrow \mathbb{R}^+$ is a decreasing right-continuous function and from \eqref{eq: cap estimate 3} we see that $g(+\infty)=0$. Also by an application of H\"older's inequality and \cite[Proposition 2.7]{GZ05} there is a constant $t_0>0$ depending only on $a, p, \|f\|_p$ such that 
\begin{equation}
\label{eq unif t}
\int_{\{u<\chi-t_0\}} f\omega^n \leq \int_{\{u<\chi-t_0\}} \frac{|\chi-u|}{t_0} f\omega^n\leq  \frac{\|f\|_p}{t_0} \left ( \int_X |u-\chi|^q \omega^n\right )^{\frac{1}{q}} \leq \frac{1-a}{(2B)^n}, 
\end{equation}
where $q>1$ is the conjugate exponent of $p$. In the last line above both $u$ and $\max(u,\chi)$ satisfy $\sup_X u =0, \sup_X \max(u,\chi)=0$, hence by \cite[Proposition 2.7]{GZ05} the constant $t_0$ can be chosen to be only dependent on $X,\theta,\omega,p, \|f\|_p, (1-a)^{-1}, B$ (but not on $u$ and $\chi$). 

It then follows from \eqref{eq: cap estimate 3} and \eqref{eq unif t} that $g(t_0+1) \leq (2B)^{-1}$. Hence from \cite[Lemma 2.4 and Remark 2.5]{Eyssidieux_Guedj_Zeriahi_JAMS_2009} it follows that $g(t_0+3)=0$. We finally conclude that $u\geq \chi-t_0-3$ almost everywhere on $X$, hence everywhere as desired. 
\end{proof}

\subsection{$\textup{Cap}_\phi$ with model potential $\phi$}

In order to use $\textup{Cap}_\chi$ in an effective manner, additional assumptions need to be made on the potential $\chi$. As in \cite{DDL2}, in this section we assume that $\chi := \phi$, where $\phi$ is a model potential and has non-collapsing mass:
$$
P[\phi]=\phi \ \textup{ and } \ \int_X \theta_{\phi}^n >0. 
$$

For elementary reasons $h_{E,\phi}^*$ is a $\theta$-psh function on $X$ which has the same singularity type as $\phi$, in fact $\phi-1\leq h_{E,\phi}^*\leq \phi$. A similar conclusion holds for  $V_{E,\phi}^*$ if $E$ is non-pluripolar, more precisely:
\begin{equation*}
\phi \leq V_{E,\phi}^*\leq \phi + M_\phi(E).  
\end{equation*}
Indeed, the first estimate is trivial, while for the second one we notice that every candidate potential of $V^*_{E,\phi} - M_\phi(E)$ is non-positive and more singular than $\phi$. Hence the supremum of all these potentials has to be less than $P[\phi]=\phi$.


\begin{lemma}\label{lemma: balayage}
Let $u\in \PSH(X, \theta)$. Let $B\subset X$ be a small  ball whose closure is contained in $\Amp(\{\theta\})$, and let $g$ be a local potential of $\theta$ in a neighborhood of $\overline{B}$. Then there exists $\hat{u}\in \psh(X, \theta)$ such that $\hat{u}=u $ on $X\setminus B$, $\hat{u}\geq u$ on $X$, $\sup_X \hat{u} \leq \sup_X u + {\rm osc}_B(g)$, and $\theta_{\hat{u}}^n(B)=0$. Moreover, if $u\leq v$ then $\hat{u}\leq \hat{v}$.
\end{lemma}

\begin{proof}
First assume that $[u]=[V_\theta]$. Then $u$ is bounded on a neighborhood of $\overline{B}$ contained inside $\textup{Amp}\{\theta\}$, hence the classical balayage method gives $\hat u \in \textup{PSH}(X,\theta)$ satisfying the required properties. This construction is monotonic in the sense that $u\leq v$ implies $\hat{u}\leq \hat{v}$.

Now assume that $u$ has arbitrary singularity type, and let $u_k := \max(u,V_\theta-k)$. Then $u_k$ decreases to $u$, and consequently $\hat u_k$ decreases to some $\hat u \in \textup{PSH}(X,\theta)$ for which  $\hat{u}=u $ on $X\setminus B$, $\hat{u}\geq u $ on $X$, and $\sup_X \hat{u} \leq \sup_X u + {\rm osc}_B(g)$.  

Lastly, by \cite[Theorem 2.3]{DDL2} we have that
$$\liminf_{k\rightarrow +\infty} \int_X\chi \theta^n_{\hat{u}^k} \geq \int_X\chi \theta^n_{\hat{u}}, $$
for all positive continuous functions $\chi: X \to \Bbb R$. This gives that $\theta_{\hat{u}}^n(B)=0$. 
\end{proof}

The following result was proved in \cite[Lemma 4.4]{DDL2} when $\phi$ has small unbounded locus. As we now show, this assumption is unnecessary: 
\begin{lemma}\label{lem: relative extremal contact set}
  If $E$ is a Borel set then $\theta_{h^*_{E,\phi}}^n$ vanishes in the open set $\{h^*_{E,\phi}<0\} \setminus \bar{E}$. 
\end{lemma}
\begin{proof}
  By Choquet's lemma there exists an increasing sequence $(u_j)$ of $\theta$-psh functions on $X$ such that $u_j\geq \phi-1$ on $X$, $u_j= \phi-1$ on $E$, $u_j\leq \phi$, and $(\lim_{j\to +\infty}u_j)^*=h_{E,\phi}^*$. If $B$ is a (very) small ball whose closure is contained in the open set $U:= \{h^*_{E,\phi}<0\}\cap \Amp(\theta) \setminus \bar{E}$ then by Lemma \ref{lemma: balayage} below there exists an increasing sequence $\widehat{u}_j$ of $\theta$-psh functions on $X$ with the following properties \vspace{0.15cm}:\\ 
\noindent(a) $0\geq \widehat{u}_j=u_j$ on $X\setminus B$, $0\geq  \widehat{u}_j \geq u_j$ on $X$,\\
\noindent (b) $\theta_{\widehat{u}_j}^n=0$ in $B$. \vspace{0.15cm}\\
Observe that by construction $\widehat{u}_j\geq  \phi-1$ on $X$ but it may be strictly less singular than $\phi$ and will not contribute to the definition of $h_{E,\phi}$.
To get around this difficulty we introduce the following functions
  \[
  v_j:= P_{\theta}[\phi](\widehat{u}_j):= \left( \lim_{C\to +\infty} P_{\theta}(\phi+C,\widehat{u}_j) \right)^*. 
  \]
 It follows from \cite[Theorem 3.8]{DDL2} that  $\theta_{v_j}^n(B) \leq \theta_{\widehat{u}_j}^n(B)=0$. Also,   
 $$\phi -1= P_{\theta}[\phi](\phi-1)\leq v_j\leq P_{\theta}[\phi](V_\theta)= \phi. $$ Thus $v_j$ has the same singularity type as $\phi$, and $v_j=\phi-1$ on $E$. Hence $v_j$ contributes to the definition of $h_{E,\phi}$.  We also have that $v_j \geq P_{\theta}[\phi](u_j)=u_j$. Therefore, $v_j$ is an increasing sequence of $\theta$-psh functions such that $(\lim_{j} v_j)^* = h_{E,\phi}^*$. Then \cite[Theorem 2.3]{DDL2} yields $\theta_{h^*_{E,\phi}}^n(B)=0$ as desired. 
\end{proof}

Lemma \ref{lemma: balayage} plays an important role in the proof of the next lemma as well:

\begin{lemma}\label{lem: global extremal contact}
  If $E$ is a non-pluripolar Borel set then $\theta_{V_{E,\phi}^*}^n$ vanishes in $X\setminus \overline{E}$. 
\end{lemma}

\begin{proof}
Since $E$ is non-pluripolar, $V_{E,\phi}^*$ is a $\theta$-psh function (as explained above). 
  By Choquet's lemma there exists an increasing sequence $(u_j)$ of $\theta$-psh functions on $X$ having the same singularity type as $\phi$ such that $u_j= \phi$ on $E$ and $(\lim_{j}u_j)^*=V_{E,\phi}^*$. By taking $\max(u_j,\phi)$ we can assume that $u_j\geq \phi$. Fix an open ball $B$ contained in the open set $U:= \Amp(\theta) \setminus \overline{E}$. By Lemma \ref{lemma: balayage} there exists an increasing sequence $\widehat{u}_j$ of $\theta$-psh functions on $X$ with the following properties: \vspace{0.15cm}\\
\noindent (a) $\widehat{u}_j=u_j$ on $X\setminus B$, $\widehat{u}_j \geq u_j$ on $X$,\\
\noindent (b) $\theta_{\widehat{u}_j}^n=0$ in $B$. \vspace{0.15cm}\\
Observe that by construction $\widehat{u}_j\geq  \phi$ on $X$, $\widehat{u}_j=\phi$ on $E$, but $\widehat{u}_j$ may be strictly less singular than $\phi$ and might not contribute to the definition of $V_{E,\phi}$. We will instead consider the projection $v_j:= P_{\theta}[\phi](\widehat{u}_j)$. 
 It follows from \cite[Theorem 3.8]{DDL2} that  $\theta_{v_j}^n(B)=0$. Since $\phi=P[\phi]$ and $\hat{u}_j\geq \phi$, it follows that $v_j= P_{\theta}[\phi](\hat{u}_j)$ has the same singularity type as $\phi$ and $v_j\geq \phi$. In addition to this, since $\hat{u}_j=u_j=\phi$ on $E$, $v_j$ contributes to the definition of $V_{E,\phi}$, implying that $v_j\leq V_{E,\phi}$.  Recall that $u_j\leq v_j$ and $(\lim u_j)^* =V_{E,\phi}^*$. Therefore, $v_j$ is an increasing sequence of $\theta$-psh functions such that $(\lim_{j} v_j)^* = V_{E,\phi}^*$. Lastly, \cite[Theorem 2.3]{DDL2} yields that $\theta_{V^*_{E,\phi}}^n(B)=0$. 
\end{proof}

The proof of the following proposition carries over from \cite[Theorem 4.5]{DDL2}:

\begin{prop}\label{prop: capi formula compact}
	If  $K$ is a compact subset of $X$ and $h:=h_{K,\phi}^*$ then 
	\[
	\capphi(K) =\int_K \theta_{h}^n =\int_X (\phi-h) \theta_h^n. 
	\]
\end{prop}
As an application of the previous result, we note the following corollary:
\begin{coro}
	\label{cor: ext capacity of compact sets}
	If $(K_j)$ is a decreasing sequence of compact sets then 
	\[
	\capphi(K) =\lim_{j\to +\infty} \capphi(K_j),
	\]
    where $K:= \bigcap_j K_j$.
	In particular, for any compact set $K$ we have 
	\[
	\capphi(K) =\inf\{\capphi(U) \setdef K\subset U\subset X\ ; \ U \ \textrm{is open in}\ X\}.
	\]
\end{coro}

\begin{proof}Let $h_j:=h^*_{K_j,\phi}$ be the relative extremal function of $(K_j,\phi)$. Then  $(h_j)$ increases almost everywhere to $h\in \psh(X,\theta)$ which satisfies $\phi-1\leq h\leq \phi$, since $\phi-1\leq h_j\leq \phi$.  Using the continuity of the Monge-Amp\`ere measure along monotonic sequences (see \cite[Theorem 2.3 and Remark 2.4]{DDL2}) we have that $\theta_{h_j}^n$ converges weakly to $\theta_h^n$. Fix $k\in \mathbb{N}$.
Since $K_k$ is compact it follows that 
 \[
 \theta_{h}^n ( K_k)\geq \limsup_{j\to +\infty} \theta_{h_j}^n( K_k).
 \]
 It then  follows from Proposition \ref{prop: capi formula compact}   that, for $k\in \mathbb{N}$ fixed,
\begin{flalign*}
	\lim_{j\to +\infty}\capphi(K_j) =  \lim_{j\to +\infty} \int_{K_j} \theta_{h_j}^n \leq \limsup_{j\to +\infty} \int_{K_k} \theta_{h_j}^n \leq  \int_{K_k} \theta_h^n. 
\end{flalign*}
Letting $k\to +\infty$ we conclude that $\lim_{j\to +\infty}\capphi(K_j)\leq \int_K \theta_h^n \leq  \capphi(K).$ Since the reverse inequality is trivial, this gives the proof of the first statement. 

	To prove the last statement, let $(K_j)$ be a decreasing sequence of compact sets such that $K$ is contained in the interior of $K_j$ for all $j$. Then by the first part of the corollary we have that 
	\begin{eqnarray*}
	\capphi(K) =\lim_{j\to +\infty} \capphi(K_j)&\geq & \lim_{j\to +\infty} \capphi(\textrm{Int}(K_j)) \\
    &\geq & \inf\{\capphi(U) \setdef K\subset U\subset X\ ; \ U \ \textrm{is open in}\ X\},
	\end{eqnarray*}
    hence equality.
\end{proof}

The Alexander-Taylor and Monge-Amp\`ere capacities are related by the following estimates, whose proof carries over from \cite[Lemma 4.9]{DDL2}:

\begin{lemma}\label{lem: compcap} Suppose $K \subset X$ is a compact subset and $\textup{Cap}_\phi(K)>0$. Then we have
$$1\leq\bigg(\frac{\int_X \theta_\phi^n}{\textup{Cap}_\phi(K)}\bigg)^{1/n}\leq \max(1,M_\phi(K)).$$
\end{lemma} 

Lastly we point out that any measure with $L^{1+\varepsilon}$ density is dominated by the relative capacity. The proof of this result also carries over verbatim from \cite[Proposition 4.30]{DDL2}: 

\begin{prop}\label{prop: capLp} 
Let $f \in L^p(X,\omega^n), \ p >1$ with $f \geq 0$.
Then there exists $C>0$ depending only on $\theta,\omega,p,X,n$ and $\|f\|_{L^p}$ such that 
$$\int_E f \omega^n \le \frac{C}{\big(\int_X \theta_\phi^n\big)^2} \cdot \textup{Cap}_\phi (E)^2$$
for all  Borel sets $E  \subset X$. 
\end{prop}

\section{Monge-Amp\`ere equations with prescribed singularity type} \label{sect: MA equation}

The goal of this section is to prove the existence and uniqueness of
 solutions to the Monge-Amp\`ere equation 
\begin{equation}\label{eq: mu_presc_eq}
\theta_u^n = \mu, \ u \in \mathcal{E}(X,\theta,\phi),
\end{equation}
where $\mu$ is a given non-pluripolar Borel measure on $X$ and $\phi$ is a $\theta$-psh function on $X$ such that 
$$P[\phi]=\phi \ \textup{ and } \ \int_X \theta_{\phi}^n= \mu(X)>0.$$ 

In the particular case when $\mu =f  \omega^n$ for some $f \in L^p(X,\omega), \ p >1$, we will show that the solution $u$ additionally satisfies $[u]=[\phi]$. 

\subsection{Construction of supersolutions with $L^p$ density}

\begin{prop}\label{prop: super solution}
	Assume that $0\leq f \in L^p(X,\omega^n)$ for some $p>1$ and $\int_X f\omega^n =\int_X \theta_{\phi}^n$. Then for each $b>1$, there exists $v\in \PSH(X,\theta)$, which is less singular than $\phi$, such that 
	\[
	\theta_{v}^n \leq bf\omega^n. 
	\]
\end{prop}

\begin{proof}
	Fix $a\in (0,1)$. For $k\in \mathbb{N}^*$ we choose $\varphi_k \in \Ec(X,\theta)$ with $\sup_X \varphi_k=0$ such that
	\[
	\theta_{\varphi_k}^n  = c_k f\omega^n + a \mathbbm{1}_{\{\phi \leq V_{\theta}-k\}} \theta_{\max(\phi,V_{\theta}-k)}^n. 
	\]
	Here, the constant $c_k\geq 1$ is chosen so that the total mass of the measure on the right-hand side is $\int_X \theta_{V_{\theta}}^n$. This insures existence (and uniqueness) of $\varphi_k$, as follows from \cite[Theorem A]{BEGZ10}. A direct computation shows that 
	\begin{equation}
		\label{eq: mass computation 1}
		\int_X\theta_{V_{\theta}}^n = c_k \int_X \theta_{\phi}^n + a\left (\int_X\theta_{V_\theta}^n-\int_{\{\phi>V_{\theta}-k\}} \theta_{\max(\phi,V_{\theta}-k)}^n \right ). 
	\end{equation}
	As a consequence, $c_k\nearrow c(a)\geq 1$ given by $c(a) = a+  (1-a) \int_X\theta_{V_{\theta}}^n/\int_X \theta_{\phi}^n.$
    
We choose $a\in (0,1)$ close enough to $1$ such that  $c(a) < b$. Fix $\varepsilon\in (0,1)$ such that $a(1-\varepsilon)^{-n} <1$. Set $\psi_k:= (1-\varepsilon) \max(\phi,V_{\theta}-k) + \varepsilon V_{\theta} $ and notice that $0\geq \psi_k \in \PSH(X,\theta)$, $\psi_k \geq \phi$.  Additionally, we notice that $\theta_{\max(\phi,V_{\theta}-k)}^n \leq (1-\varepsilon)^{-n} \theta_{\psi_k}^n,$ in particular
	\[
	\theta_{\varphi_k}^n  \leq  c(a) f\omega^n + \frac{a}{(1-\varepsilon)^n} \theta_{\psi_k}^n. 
	\]
Since $f\in L^p(X,\omega^n), p>1$, it follows from \cite[Proposition 4.3]{BEGZ10} that 
\[
\int_E f\omega^n \leq A_1 [\textup{Cap}_{V_\theta}(E)]^2,
\]
for every Borel set $E\subset X$, where $A_1$ is a positive constant depending on $\theta,n,p,\|f\|_p$. It then follows from Lemma \ref{lem: cap psi vs cap theta} below that 
\begin{equation}
	\label{eq: cap estimate 2}
	\int_E f\omega^n \leq \frac{A_1}{\varepsilon^{2n}} [\textup{Cap}_{\psi_k}(E)]^2,
\end{equation}
for every Borel set $E\subset X$.
Moreover, it follows from \cite[Theorem 1.2]{DDL1} that $P_{\theta}[\varphi_k]=V_{\theta}$. Hence we can apply Theorem \ref{thm: uniform estimate} to $\chi: = \psi_k$, $u := \varphi_k$, and $\tilde a := a(1-\varepsilon)^{-n}$ to conclude that
	\[
	0 \geq \varphi_k \geq \psi_k-C\geq \phi -C, 
	\]
	for all $k$. Here, $C>0$ depends on $\varepsilon, a, A_1, p, \|f\|_p$. Now, for each $k,j$ we set 
	\[
	v_{k,j}:= P_{\theta}(\min(\varphi_k,..., \varphi_{k+j})).
	\]
	Observe that $0 \geq v_{k,j}\geq \phi -C$, for every $j,k$. Consequently $v_{k,j} \searrow v_k \in \textup{PSH}(X,\theta)$ as $j \to \infty$, $v_k \nearrow v \in \textup{PSH}(X,\theta)$ as $k \to \infty$, and the following estimates trivially hold:
$$0 \geq v_{k},v\geq \phi -C.$$
In addition to the above, observing that $\{\phi\leq V_\theta-k-\ell\}\subset \{\phi\leq V_\theta-k\}$ for any $\ell=0, \dots, j$, it follows from Lemma \ref{lem: MA of rooftop envelope} that 
	\begin{equation}
		\label{eq: envelope of super solutions 1}
		\theta_{v_{k,j}}^n \leq c(a) f\omega^n + \mathbbm{1}_{\{\phi\leq V_{\theta}-k\}} \sum_{\ell=0}^j \theta_{\max(\phi,V_{\theta}-k-\ell)}^n. 
	\end{equation}
Now fix $s>0$ and consider $k >s$. As a result of the above estimate, for arbitrary $\delta \in (0,1)$ we can write:
	\[
\frac{\max(\phi-V_{\theta}+s,0)}{\max(\phi-V_{\theta}+s,0)+\delta} \cdot \theta_{v_{k,j}}^n	\leq \mathbbm{1}_{\{\phi>V_{\theta}-s\}} \theta_{v_{k,j}}^n \leq \mathbbm{1}_{\{\phi>V_{\theta}-k\}} \theta_{v_{k,j}}^n\leq c(a)f\omega^n \leq b f \omega^n. 	\]
Since the fraction on the left hand side is a bounded quasi-continuous function (with values in $[0,1]$), we can apply \cite[Theorem 2.3]{DDL2} to conclude that 
	\[
	\frac{\max(\phi-V_{\theta}+s,0)}{\max(\phi-V_{\theta}+s,0)+\delta} \cdot \theta_{v_{k}}^n \leq bf\omega^n. 	\]
Another application of \cite[Theorem 2.3]{DDL2} yields that 
	\[
	\frac{\max(\phi-V_{\theta}+s,0)}{\max(\phi-V_{\theta}+s,0)+\delta} \cdot \theta_v^n \leq bf\omega^n. 	\]
Now letting $\delta \searrow 0$ we arrive at $\mathbbm{1}_{\{\phi>V_{\theta}-s\}} \theta_{v}^n \leq b f\omega^n.$ Finally, letting $s\to +\infty$ the conclusion follows. 
\end{proof}

We provide the following lemma that was used in the proof of the above proposition:

\begin{lemma}\label{lem: cap psi vs cap theta}
Suppose $\varepsilon \in (0,1)$, $w\in \PSH(X,\theta)$, $w\leq 0$, and $\psi := (1-\varepsilon) w + \varepsilon V_{\theta}\leq 0.$ Then for any Borel subset $E\subset X$ one has 
	\[
	{\rm Cap}_{\theta}(E):=\textup{Cap}_{V_\theta}(E) \leq \varepsilon^{-n} \cappsi(E). 
	\]
\end{lemma}
\begin{proof}
	If $u\in \PSH(X,\theta)$ satisfies $V_{\theta}-1\leq u \leq V_{\theta}$ then the function $v:=(1-\varepsilon) w + \varepsilon u$ is $\theta$-psh and satisfies $\psi -1 \leq v \leq \psi$, hence 
	\[
	\varepsilon^n \int_E \theta_u^n  \leq \int_E ((1-\varepsilon)\theta_w + \varepsilon \theta_u)^n =  \int_E \theta_v^n \leq \cappsi(E). 
	\]
	Taking the supremum over such $u$ one concludes the proof.
\end{proof}
\subsection{Existence for measures with bounded density}

\begin{theorem}\label{thm: existence L1 density}
		Assume that $0\leq f \in L^{\infty}(X,\omega^n)$ and $\int_X f\omega^n =\int_X \theta_{\phi}^n$. Then there exists  $u\in \Ec(X,\theta,\phi)$ such that $\theta_u^n = f\omega^n$.  
\end{theorem}

\begin{proof}
For each $k\in \mathbb{N}^*$ it follows from Proposition \ref{prop: super solution} that there exists $\varphi_k \in \PSH(X,\theta)$, normalized by $\sup_X \varphi_k=0$, such that $\theta_{\varphi_k}^n \leq (1+2^{-k})f\omega^n$ and $\varphi_k$ is less singular than $\phi$. In particular $P_\theta[\varphi_k]$ is less singular than $\phi$. It follows from Theorem \ref{thm: uniform estimate} (with $a=0$) and Proposition \ref{prop: capLp} that $\varphi_k\geq \phi -C$, for some uniform constant $C>0$. 

As in  the proof of Proposition \ref{prop: super solution}, we set $v_{k,j}= P_\theta(\min(\varphi_k, \dots \varphi_{k+j}))$. We then have $v_{k,j} \searrow v_k \in \textup{PSH}(X,\theta)$ as $j \to \infty$, $v_k \nearrow \varphi \in \textup{PSH}(X,\theta)$ as $k \to \infty$ and $0\geq \varphi \geq \phi-C$. Moreover, by Lemma \ref{lem: MA of rooftop envelope} we get that $\theta_{v_{k,j}}^n\leq (1+2^{-k}) f\omega^n$. Using \cite[Theorem 2.3]{DDL2} we arrive at 
$$ \theta_{v_{k}}^n\leq (1+2^{-k}) f\omega^n.$$ Another application of \cite[Theorem 2.3]{DDL2} gives $ \theta_{\varphi}^n\leq f\omega^n.$
It follows from \cite[Theorem 1.2]{WN17}  that $\int_X \theta_{\varphi}^n \geq \int_X \theta_{\phi}^n = \int_X f\omega^n$, hence we actually have $\theta_{\varphi}^n = f\omega^n$. 
Given our normalizations, we get that $\varphi$ is a candidate in the definition of $P[\phi]$, hence $\varphi \leq P[\phi]=\phi.$
This means that $\varphi$ has the same singularity type of $\phi$ and that in particular $\varphi\in \Ec(X,\theta,\phi)$.
\end{proof}

\subsection{Existence for non-pluripolar measures} \label{subsect: CY equation non pluripolar}
 Following the strategy in \cite{GZ07} (going back to \cite{Ce98}) we will now solve \eqref{eq: mu_presc_eq}. We first describe the technical setup.  
 
 Let $\Omega_{\alpha}$, $\alpha =1,...,N$ be a finite covering of  $X$ by open balls which are biholomorphic to the unit ball in $\mathbb{C}^n$ via $\tau_{\alpha}: B\rightarrow \Omega_\alpha$. Let $\chi_{j}$ be spherically symmetric smoothing kernels in $\mathbb{C}^n$ approximating the Dirac mass concentrated at the origin. Let $(\rho_{\alpha})_{\alpha=1}^N$ be a partition of unity subordinate to $(\Omega_{\alpha})_{\alpha=1}^N$. Let $\mu_{\alpha}$ be the pullback of $\mu\mid_{\Omega_{\alpha}}$ by the biholomorphism $\tau_{\alpha}$, which is a positive Borel measure in the unit ball $B$ in $\mathbb{C}^n$. For each $j$ we define a (smooth) measure on $X$,
 \[
 \nu_j := c_j \sum_{\alpha} \rho_{\alpha} \cdot (\tau_{\alpha})_{*}(\mu_{\alpha}  \star \chi_{j}),
 \]
 where $c_j$ is a positive normalization constant insuring that $\nu_j(X)=\int_X \theta_{\phi}^n>0$.  Since $\nu_j\to \mu$ weakly it follows that $c_j\to 1$, hence we can assume that $c_j\leq 2$ for all $j$.  

We will need the following lemma.

\begin{lemma}\label{lem: convergence}
Assume $\mu$ is a non-pluripolar measure on $X$. Let $u_j, u\in \PSH(X, A\omega)$ for some $A>0$. Assume $u_j\rightarrow u$ in $L^1(X, \omega^n)$ 
and  $\sup_j \int_X |u_j|^2 d\mu<+\infty$. Then $$\int_X |u_j-u|d\mu \rightarrow 0.$$
\end{lemma}

\begin{proof}
It follows from \cite[Lemma 11.5]{GZ17}  that  
	\begin{equation}
		\label{eq: Cegrell}
		\int_{X} (u_j-u) d\mu \to 0. 
	\end{equation}
    For each $j>0$ we set $\tilde{u}_j:= (\sup_{k\geq j} u_k)^*$. Then $\tilde{u}_j \in \PSH(X,\theta)$ and $\tilde{u}_j$ decrease to $u$ pointwise. Since $\tilde{u}_j \geq \max(u_j,u)$ we can write 
    \[
    |u_j-u|  = 2 \max(u_j,u) -u_j-u \leq  2(\tilde{u}_j-u) +(u-u_j). 
    \]
    It thus follows from the monotone convergence theorem  and \eqref{eq: Cegrell} that $\int_X |u_j-u|d\mu \leq 2\int_X (\tilde{u}_j -u) d\mu + \int_X (u-u_j)d\mu \to 0.$  
\end{proof}

\begin{lemma}
	\label{lem: preparation 1}
	Assume that $\mu = f\id_{\{\phi>-A\}} \omega_{v}^n$ where $v\in \PSH(X,\omega)\cap L^{\infty}(X)$, $A>0$ is a constant, $f \geq 0$ is bounded. Assume also that $\mu \leq B \capphi$ for some positive constant $B$. Then there exists $u\in \Ec(X,\theta,\phi)$ such that $\theta_u^n =\mu$.  
\end{lemma}
\begin{proof}For each $j$ let $u_j\in \Ec(X,\theta,\phi)$ be such that $\sup_X u_j=0$ and $\theta_{u_j}^n =\nu_j$. These potentials exist by Theorem \ref{thm: existence L1 density}. Up to extracting a subsequence we can assume that $u_j\to_{L^1}u\in \PSH(X,\theta)$. Consequently, $u\leq \phi$. The goal is to prove that $u\in \Ec(X,\theta,\phi)$ and $\theta_u^n =\mu$. As we will see, the crucial ingredient is showing that $\int_X |u_j-u|\theta_{u_j}^n \to 0$.  We proceed in several steps. For notational convenience, we will use $C>0$ to denote various uniform constants independent of $j$, and we will also omit $\tau_{\alpha}$ from the formulas, as this will not cause confusion. 

\medskip\noindent\textbf{Step 1.} We claim that $\int_X |u_j-\phi| \theta_{u_j}^n$ is bounded.   
    
    Let $K_{\alpha}$ be a compact subset of $\Omega_{\alpha}$ such that $\Supp(\rho_{\alpha}) \Subset K_{\alpha}$.  Using the fact that $0\leq \rho_{\alpha}\leq 1$, $\phi\leq 0$ and the definition of the convolution we get, for $j$ large enough,
	\begin{flalign*}
	\int_X (\phi-u_j) d\nu_j & \leq C \sum_\alpha \int_{\Omega_{\alpha}}  \rho_{\alpha}(\phi-u_j)  d(\mu \star \chi_j) \leq C \sum_\alpha \int_{\Supp(\rho_{\alpha})}  (\phi-u_j) d(\mu \star \chi_j) \\
	&\leq C\sum_\alpha \int_{\Supp(\rho_{\alpha})} (-u_j) d(\mu \star \chi_j) \leq C\sum_\alpha \int_{K_{\alpha}} (-u_j \star \chi_j) d\mu,
	\end{flalign*}
where the last inequality follows from \cite[Theorem 1.1.5(v)]{Blocki}.
Note also that since $u_j$ is quasi-psh we have $u_j \star \chi_j \geq u_j -C$ on $K_\alpha$. The latter follows from the fact that $u_j= \varphi_j- g$, where $\varphi_j$ is psh on $\Omega_\alpha$ and $g$ is the local potential of $\theta$ in $\Omega_\alpha$; also, the mean value inequality for psh functions together with the fact that $\chi_j$ are radial functions give $\varphi_j\star \chi_j \geq \varphi_j$ on $K_{\alpha}$, for $j$ large enough.
We thus get
	\[
	\int_X (\phi-u_j) d\nu_j \leq C\left(\int_X |u_j| d\mu +1 \right).
	\]
Since $\mu \leq C \omega_v^n$ and $\sup_X u_j=0$ it follows from \cite[Corollary 3.3]{GZ05} that the right-hand side above is uniformly bounded in $j$, hence $\int_X |u_j-\phi|\theta_{u_j}^n=\int_X (\phi-u_j) d\nu_j\leq C.$
    
\medskip \noindent\textbf{Step 2.} We prove that $\int_X |u_j-u|d\mu \to 0$.

	Since $\mu \leq B \capphi$ it follows from Lemma \ref{lem: energy capacity estimate} and Step 1 that 
	 \[
	 \int_X |u_j-\phi|^2 d\mu \leq C \left(\int_X|u_j-\phi| \theta_{u_j}^n+1\right) \leq C'. 
	 \]
	 Since $\mu$ is supported on $\{\phi>-A\}$ it follows that 
	 $$
	 \int_X |u_j|^2 d\mu \leq 2\int_X |u_j-\phi|^2 d\mu + \int_X |\phi|^2 d\mu \leq 2C' + 2A^2\mu(X).
	 $$
	 Lemma \ref{lem: convergence} then gives the conclusion.

\medskip\noindent\textbf{Step 3.} We prove that $\int_X |u_j-u| \theta_{u_j}^n \to 0$. 

It suffices to argue that $\int_X \rho_{\alpha}|u_j-u| \theta_{u_j}^n \to 0$ for each $\alpha$. Let $\varphi_j=u_j+g$, $\varphi=u+g$, where $g$ is a local smooth potential of $\theta$ in $\Omega_{\alpha}$. For each $k$ we set $\tilde{\varphi}_j := (\sup_{k\geq j} \varphi_k)^*$. Then $\tilde{\varphi}_j$ decrease to $\varphi$ and we have $|\varphi_j-\varphi|\leq 2(\tilde{\varphi}_j-\varphi) + (\varphi-\varphi_j)$. Observing that $\nu_j \leq 2 (\mu \star \chi_j)$ on $\Supp(\rho_{\alpha})$, we then have 
	\begin{flalign*}
		\int_X \rho_{\alpha} |u_j-u| d\nu_j & \leq 2 \int_{\Supp(\rho_{\alpha})} |\varphi_j-\varphi|  d(\mu \star \chi_j) \leq  2\int_{K_{\alpha}} |\varphi_j-\varphi| \star \chi_j \, d\mu\\
		& \leq 2\int_{K_{\alpha}}( 2(\tilde{\varphi}_j-\varphi) + (\varphi-\varphi_j))\star \chi_j\,  d\mu.
	\end{flalign*}
	Since $\varphi,\varphi_j, \tilde{\varphi}_j$ are psh in $\Omega_{\alpha}$ (and in particular $\varphi \star \chi_j \geq \varphi$) it follows that
	\[
	\int_{K_{\alpha}}  (\tilde{\varphi}_j-\varphi) \star \chi_j d\mu \leq  \int_{K_{\alpha}} ( \tilde{\varphi}_j \star \chi_j -\varphi) d\mu \to 0 
	\]
	by monotone convergence . For the second term we have 
	\begin{flalign*}
	\int_{K_{\alpha}}  (\varphi-\varphi_j) \star \chi_j d\mu & \leq \int_{K_{\alpha}}  (\varphi\star \chi_j - \varphi_j)d\mu \leq \int_{K_{\alpha}}  (\varphi\star \chi_j -\varphi + \varphi - \varphi_j)d\mu \\
   & \leq \int_{K_{\alpha}}  (\varphi\star \chi_j -\varphi) d\mu +\int_{K_{\alpha}} |u - u_j|d\mu \to 0
	\end{flalign*}
	as follows from the monotone convergence theorem and Step 2. 
    
\medskip\noindent\textbf{Step 4.} It follows from Step 1 and Lemma \ref{lem: BEGZ phi} that $u\in \Ec(X,\theta,\phi)$ (in fact $u\in \Ec^1(X,\theta,\phi)$). We next show that $\theta_u^n=\mu$. By step 3, up to extracting a subsequence we can assume that
\begin{equation}\label{eq: E1_ujest}
\int_X |u_j-u|\theta_{u_j}^n \leq 2^{-j}, \ j \in \Bbb N.
\end{equation} 
We set $h_j := \max(u_j,u-1/j)$. Then, by \cite[Lemma 1.2]{GLZ17}, $h_j \to u$ in capacity and \cite[Theorem 2.3 and Remark 2.5]{DDL2} gives that $\theta_{h_j}^n\to \theta_u^n$ weakly. Set $\eta_j:= \id_{\{u_j\leq u-1/j\}} \theta_{u_j}^n$. By the locality of the complex Monge-Ampere measure with respect to the plurifine topology we have 
	\[
	 \theta_{h_j}^n + \id_{\{u_j\leq u-1/j\}} \theta_{u_j}^n  \geq \theta_{u_j}^n. 
	\]  
From \eqref{eq: E1_ujest} we get that $\eta_j(X)= \int_{\{u_j\leq u-1/j\}} \theta_{u_j}^n \leq \int_X j |u_j-u| \theta_{u_j}^n\leq j2^{-j} \to 0$. It thus follows that $\eta_j$ converges weakly to $0$, hence $\theta_u^n \geq \lim_j \theta_{u_j}^n=\lim_j \nu_j=\mu $. After comparing the total masses (via \cite[Theorem 1.2]{WN17}),  we have that $\theta_u^n = \mu$. 
\end{proof}

\begin{prop}
	\label{prop: mu dominated by capphi}
	Assume that $\mu \leq B \capphi$ for some positive constant $B$. Then there exists $u\in \Ec(X,\theta,\phi)$ such that $\theta_u^n=\mu$. 
\end{prop}
\begin{proof}
It follows from \cite[Theorem A]{BEGZ10} that $\mu = c \omega_{\varphi}^n$  for some $\varphi\in \Ec(X,\omega)$, $\sup_X \varphi=0$ and $c=  \left(\int_X \theta_\phi^n\right) \left( \int_X \omega^n\right)^{-1} >0$. By considering $v:= e^{\varphi}$ which is a bounded $\omega$-psh function  and noting that $\omega_{v}^n \geq e^{n\varphi} \omega_{\varphi}^n$ we can write $\mu = f \omega_v^n$, where  $f\in L^1(X,\omega_v^n)$. Now, for each $j>0$ we set $\mu_j:= c_j\min(f,j) \id_{\{\phi>-j\}} \omega_v^n$, where $c_j>0$ is a normalization constant. Then $c_j \to 1$ as $j\to +\infty$ thus we can assume that $c_j\leq 2$ for all $j$. Note also that $\mu_j \leq 2B \capphi$. It follows from Lemma \ref{lem: preparation 1} that there exists $u_j \in \Ec(X,\theta,\phi)$, $\sup_X u_j=0$ such that $\theta_{u_j}^n =\mu_j$. Up to extracting a subsequence we can assume that $u_j \to u \in \PSH(X,\theta)$ in $L^1(X,\omega^n)$ and $u\leq \phi$.  It follows from Lemma \ref{lem: stability of subsolutions} that $\theta_u^n \geq \mu$. We finally invoke \cite[Theorem 1.2]{WN17} to obtain that $\int_X \theta_u^n= \mu(X)$. Hence the conclusion.
\end{proof}

\begin{theorem}\label{thm: non pluripolar existence}
	Assume that $\mu$ is a non-pluripolar positive measure on $X$. Then there exists a unique $u\in \Ec(X,\theta,\phi)$ such that $\theta_u^n =\mu$ and $\sup_X u=0$.
    
In addition to this, in the particular case when $\mu = f \omega^n$ with $f \in L^p(X,\omega^n), \ p > 1$ we have that
$$\phi - C\Big(p,\omega,\int_X \theta_\phi^n, \| f\|_{L^p}\Big) \leq u \leq \phi\leq 0.$$    
\end{theorem}
\begin{proof} It follows from the arguments in \cite[Lemma 4.17]{DDL2} and Corollary \ref{cor: ext capacity of compact sets} that the set $\mathcal{M}_1$ of probability measures $\nu$ on $X$ such that $\nu\leq  \textup{Cap}_\phi$, is compact and convex. The arguments in \cite[Lemma 4.26]{DDL2} then ensure that $\mu = f \nu$, where $\nu\in \mathcal{M}_1$ and $f \in L^1(X,\nu)$.

According to the previous proposition, for $j \in \Bbb N$ we can find $u_j \in \mathcal E(X,\theta,\phi)$ such that $\sup_X u_j =0$, $u_j \leq \phi\leq 0$ and 
$$\theta_{u_j}^n = c_j \min(f,j) \nu,$$
where $c_j \geq 1$ is arranged so that $\mu(X) = c_j \int_X \min(f,j) \nu$. Hence $c_j\rightarrow 1$. After possibly taking a subsequence, we can assume that $u_j \to_{L^1} u \in \textup{PSH}(X,\theta)$, where $u \leq \phi \leq 0$, $\sup_X u =0$. Finally, Lemma \ref{lem: stability of subsolutions} implies that $\theta_u^n \geq f\nu=\mu$. Since $u \leq \phi \leq 0$, \cite[Theorem 1.1]{WN17} gives that in fact $\theta_u^n = \mu$. The uniqueness is recalled in the next result. 

The last statement follows from Theorem \ref{thm: uniform estimate} in the particular case when $a=0$. The latter theorem can be indeed applied thanks to Proposition \ref{prop: capLp} and the fact that \cite[Theorem 1.3]{DDL2} ensures $P_\theta[u]=P_\theta[\phi]=\phi$.
\end{proof}

Following Dinew  \cite{Dw09}, it is well known in pluripotential theory that existence of full mass solutions implies their uniqueness. The proof of the following theorem is exactly the same as that of \cite[Theorem 4.29]{DDL2}:

\begin{theorem} \label{thm: uniqueness}
Suppose $u,v\in \mathcal{E}(X,\theta,\phi)$ satisfy $\theta_u^n=\theta_v^n$. Then $u-v$ is constant.
\end{theorem}
Next we point out that in case $\int_X \theta_{\phi}^n = 0$ the above uniqueness result fails, even if $\phi$ is a model potential: 
\begin{remark}\label{rem: non-unique_zero_mass} Consider $\CPP^1\times \CPP^1$ equipped with the K\"ahler form $\omega = \pi_1^{*}\omega_{FS}+ \pi_2^{*}\omega_{FS}$. Using \cite[Remark 3.3]{DDL2} it is possible to find two model potentials $\alpha ,\beta \in \textup{PSH}(X,\omega)$ such that $\alpha$ is strictly less singular than $\beta$, and $\int_X \omega_\alpha^n = \int_X \omega_\beta^n =0$ (indeed, just take $\alpha := \lim_{\varepsilon \to 0} P[(1-\varepsilon)\pi_2^* v]$ and $\beta := \lim_{\varepsilon \to 0} P[(1-\varepsilon)\phi]$ in the cited example).

In particular, there exists $C>0$ such that $\gamma:=\max(\alpha,\beta+C)$ has the same singularity type as $\alpha$, but $\alpha-\gamma$ is non-constant. But  since $[\alpha]=[\gamma]$,  \cite[Theorem 1.2]{WN17} gives that $\int_X\omega_\alpha^2=\int_X \omega_\gamma^2=0$, hence $\omega_\alpha^2=\omega_\gamma^2=0$.
\end{remark}

\section{The Aubin-Yau equation}\label{sect: AY equation}
With Theorem \ref{thm: non pluripolar existence} in hand, as in \cite[Theorem 6.1]{BEGZ10}, using Schauder's fixed point theorem we will  solve the following Aubin-Yau equation: 
\begin{equation}
	\label{eq: AY}
	\theta_u^n = e^{\lambda u} \mu, \ u \in \Ec(X,\theta,\phi),
\end{equation}
where $\lambda >0$. We recall the version of Schauder's fixed point theorem that we will need:
\begin{theorem}[Schauder] \label{thm: Schauder}
Let $X$ be a Banach space, and let $K \subset X$
be a non-empty, compact and convex set. Then given any continuous mapping $F: K \to K$ there exists $x \in K$ such that $F(x)=x$.
\end{theorem}
We refer the reader to \cite[Theorem B.2, page 302]{Ta11} for a proof.

\begin{lemma}\label{lem: compact convex}
Let $\phi\in \PSH(X,\theta)$ be such that $P[\phi]=\phi$ and $\int_X \theta_{\phi}^n>0$. Then the set $F:= \{u \in \PSH(X,\theta,\phi) \setdef \int_X u \, \omega^n =0\}$ is a compact convex subset of $L^1(X,\omega^n)$.
\end{lemma}
\begin{proof}
The convexity is clear. Since the set of normalized (by $\int_X u\, \omega^n =0$) $\theta$-psh functions is compact in the $L^1$-topology (see \cite[Proposition 2.7]{GZ05}) it suffices to prove that $F$ is closed. For this purpose let   $(u_j)$ be a sequence of functions in $F$ which converges in $L^1(X,\omega^n)$ to $u\in \PSH(X,\theta)$. We want to prove that $u\in F$.  Since $\int_X u\, \omega^n =0$ it suffices to show that $u\preceq \phi$. Again, \cite[Proposition 2.7]{GZ05} yields  $\sup_X u_j\leq C$ for a uniform constant $C>0$. It follows that $u_j -C$ is a candidate defining $P[\phi]$, hence $u_j-C\leq P[\phi] = \phi$. Letting $j\to +\infty$, we arrive at $u\leq \phi+C$, finishing the proof.
\end{proof}
\begin{theorem}
	\label{thm: AY} Let $\phi \in \textup{PSH}(X,\theta)$ such that $\phi = P[\phi]$ and $\int_X \theta_\phi^n >0$. Assume that $\mu$ is a non-pluripolar positive Borel measure on $X$ satisfying $0< \int_X d \mu <+\infty$. Then there exists a unique $v\in \Ec(X,\theta,\phi)$ solving \eqref{eq: AY}.
    
    In addition to this, in the particular case when $\mu = f \omega^n$ with $f \in L^p(X,\omega^n), \ p > 1,$ we have that
$$\phi - C \leq v \leq \phi+C,$$
where $C:=C\Big(\lambda, p,\omega,\int_X \theta_\phi^n, \| f\|_{L^p}\Big)>0$.
\end{theorem}

\begin{proof}For simplicity we only treat the case $\lambda =1$, as the proof of the general result is the same. We recall that the set $F$ defined in Lemma \ref{lem: compact convex}  is a compact convex subset of $L^1(X,\omega^n)$. 
    
    We consider the map $m: F \rightarrow F$ defined as $m(\varphi) =u$ where $u\in \Ec(X,\theta,\phi)$ is the unique function such that $\int_X u\, \omega^n=0$ and   $\theta_u^n = c(\varphi)e^{\varphi} \mu$, where $c(\varphi)$ is a positive constant such that $c(\varphi)\int_X e^{\varphi} d\mu =\int_X \theta_{\phi}^n$. Then $m$ is well-defined thanks to Theorem \ref{thm: non pluripolar existence} and Theorem \ref{thm: uniqueness}. 
    
    We prove that $m$ is continuous on $F$. Assume that $(\varphi_j)$ is a sequence in $F$ which converges in $L^1(X,\omega^n)$ to $\varphi\in F$. We want to prove that $m(\varphi_j)$ converges in $L^1(X,\omega^n)$ to $m(\varphi)$. Since the sequence $(m(\varphi_j))_j$ is contained in a compact set it suffices to prove that any cluster point of this sequence is $m(\varphi)$. For this purpose, after extracting a subsequence we can assume that $m(\varphi_j)$ converges in $L^1(X,\omega^n)$ to $u\in F$. The goal is to prove that $\theta_u^n = c(\varphi) e^{\varphi} \mu$. From the normalization $\int_X \varphi_j \, \omega^n=0$ we obtain a uniform bound for $\sup_X \varphi_j$ (see \cite[Proposition 2.7]{GZ05}). It then follows that $e^{\varphi_j}$ is uniformly bounded, hence by \cite[Lemma11.5]{GZ17} we have $\int_X e^{\varphi_j}d\mu \rightarrow \int_X e^{\varphi}d\mu$. It thus follows that 
 $$c(\varphi_j)= \int_X\theta_{\phi}^n \cdot \left( \int_X e^{\varphi_j}  d\mu \right)^{-1}\to \int_X \theta_{\phi}^n \cdot  \left( \int_X e^{\varphi} d\mu\right)^{-1} =c(\varphi).$$
It thus follows from Lemma \ref{lem: convergence} together with Lemma \ref{lem: stability of subsolutions} that $\theta_u^n \geq c(\varphi) e^{\varphi} \mu$. Since $u \preceq \phi$ we must have equality because of  \cite[Theorem 1.2]{WN17}.  Hence $u=m(\varphi)$ and the continuity of $m$ is proved. It thus follows from Schauder's fixed point theorem (Theorem \ref{thm: Schauder})  that there exists $u\in F$ such that $m(u)=u$. This means that $u+\log c(u) \in \Ec(X,\theta,\phi)$ solves \eqref{eq: AY}. 
    Uniqueness is a consequence of the next result. 
    
    To argue the last statement, suppose that $\mu = f\omega^n, \ f \in L^p(X,\omega^n), p>1$. For simplicity we can assume that $\int_X \theta_\phi^n = \int_X f\omega^n = 1$. 
    
    Suppose that $\theta_u^n = e^{u}f\omega^n$ for some $u \in \Ec(X,\theta,\phi)$. First we argue that $\sup_X u$ is under control. By comparing the total masses we get that $\sup_X u \geq 0$. By compactness, we have that given $q \geq 1$, there exists $C:=C(q) >0$ such that  $\int_X|v|^q \omega^n \leq C$ for all $v \in \textup{PSH}(X,\theta)$, with $\sup_X v =0$. Using Jensen's and H\"older's inequality we obtain that:
\begin{flalign*}
0= \log \int_X \theta_u^n & =\sup_X u  + \log \int_X e^{u-\sup_X u} f \omega^n \\
& \geq \sup_X u +\int_X (u-\sup_X u)\, f \omega^n \geq \sup_X u -C^{\frac{1}{q}}\| f \|_p,
\end{flalign*}
where $q =\frac{p}{p-1}$. Consequently, $0 \leq \sup_X u \leq C^{\frac{1}{q}}\| f \|_p$. Since $\theta_u^n \leq e^{\sup_X u} f \omega^n$, Theorem \ref{thm: uniform estimate} implies that 
$$\phi - C\Big( p,\omega,\int_X \theta_\phi^n, \| f\|_{L^p}\Big) \leq u - \sup_X u\leq \phi\leq 0.$$
Given that  $\sup_X u$ is also under control, the required estimate follows.
\end{proof}

Lastly, we provide a comparison principle for \eqref{eq: AY}, whose proof carries over word for word from \cite[Lemma 4.24]{DDL2}:
\begin{lemma}\label{lem: comparison}
	\label{lem: comp_sol_subsol}
	Let $\lambda>0$. Assume that $\varphi\in \Ec(X,\theta,\phi)$ is a solution of \eqref{eq: AY}, while $\psi\in \Ec(X,\theta,\phi)$ satisfies $\theta_{\psi}^n \geq e^{\lambda \psi}\mu.$
	Then $\varphi\geq \psi$ on $X$. 
\end{lemma}

\section{Log concavity of volume} \label{sect: log concave}

In this short section, we give the proof of our main application, which is a direct consequence of solvability of complex Monge-Amp\`ere equations with prescribed singularity type:\begin{theorem}\label{thm: log concavity}
	Let $T_1,...,T_n$ be closed positive $(1,1)$-currents  on a compact K\"ahler manifold $X$. Then 
\begin{equation}\label{eq: log_ineq}
	\int_X \langle T_1 \wedge ...\wedge T_n\rangle \geq \bigg(\int_X \langle T_1^n\rangle \bigg)^{\frac{1}{n}} ... \bigg(\int_X \langle T_n^n\rangle \bigg)^{\frac{1}{n}}.
\end{equation}
	In particular, $T\mapsto \left ( \int_X \langle T^n \rangle\right )^{\frac{1}{n}}$ is concave on the set of positive currents, and so is the map $T\mapsto \log \left ( \int_X \langle T^n \rangle\right )$. 
\end{theorem}

\begin{proof}We can assume that the classes of $T_j$ are big and their masses are non-zero. Other\-wise the right-hand side of the inequality to be proved is zero.  In fact, after rescaling, we can assume that $\int_X \omega^n=\int_X \langle T^n_j \rangle=1, \ j \in \{ 1,\ldots,n\}$.

Consider smooth closed $(1,1)$-forms $\theta^j$,  and $u_j \in \psh(X,\theta^j)$  such that $T_j = \theta^j_{u_j}$.  We know from \cite[Theorem 3.12]{DDL2} that $P_{\theta^j}[u_j]$ is a model  potential. 

For each $j=1,...,n$  Theorem \ref{thm: non pluripolar existence} insures existence of $\varphi_j\in \mathcal{E}(X,\theta^j,P_{\theta}[u_j])$ such that $\big(\theta^j_{\varphi_j}\big)^n = \omega^n$. 
A combination of \cite[Proposition 2.1]{DDL2} and \cite[Theorem 2.2]{DDL2} then gives
$$
\int_X \theta^1_{\varphi_1} \wedge ... \wedge \theta^n_{\varphi_n}  = \int_X  \theta^1_{P_{\theta^1}[u_1]} \wedge ... \wedge \theta^n_{P_{\theta^n}[u_n]}= \int_X  \theta^1_{u_1} \wedge ... \wedge \theta^n_{u_n} = \int_X\langle T_1 \wedge .... \wedge T_n \rangle .
$$
Finally, an application of  \cite[Proposition 1.11]{BEGZ10} (building on \cite{Diw09a}) gives $\theta^1_{\varphi_1} \wedge \ldots \wedge \theta^n_{\varphi_n} \geq \omega^n$, finishing the proof of \eqref{eq: log_ineq}. 

To prove that $T\mapsto \left ( \int_X \langle T^n \rangle\right )^{{1}/{n}}$ is concave, we take closed positive $(1,1)$-currents $T_1,T_2$ and we argue that
$$\bigg(\int_X \langle ((1-t)T_1 + tT_2)^n \rangle\bigg)^{\frac{1}{n}} \geq (1-t) \bigg(\int_X \langle T_1^n \rangle\bigg)^{\frac{1}{n}}+t \bigg(\int_X \langle T_2^n \rangle\bigg)^{\frac{1}{n}}, \ \ \ t \in [0,1].$$
However this follows, as multilinearity allows to expand the left hand side, and multiple application of \eqref{eq: log_ineq} yields the desired inequality.
Finally,  concavity of $T\mapsto \log \left ( \int_X \langle T^n \rangle\right )$ is equivalent to that of $T\mapsto \left ( \int_X \langle T^n \rangle\right )^{{1}/{n}}$ since the latter is homogeneous of degree $1$.
\end{proof}

\section{Relation to convex geometry}  \label{sect: real MA}

The goal of this section is to point out a strong connection between our $\phi$-relative pluripotential theory and ``$P$-convex geometry''. This latter subject has been explored recently in \cite{BB13, Ba17, BBL18}, motivated by the study of K\"ahler-Ricci solitons, Bergman measures, and Fekete points. We closely follow the terminology of \cite{Ba17}.

A convex body $P\subset \mathbb{R}^n$ is a compact convex subset with non-empty interior. We introduce $h_P: \mathbb{R}^n \rightarrow \mathbb{R}$ , the \emph{support function} of $P$:
$$ h_P(x)=\sup_{p\in P} \,\langle x,p\rangle.$$
This is a  homogeneous convex function, i.e. $h_P(tx)=th_P(x)$ for $t>0$. 

Following \cite{BB13}, we let $\mathcal{L}_P(\mathbb{R}^n)$ denote the space of all convex functions $h$ on $\mathbb{R}^n$ such that $h\leq h_P+ C$, for some constant $C\in \Bbb R$. We recall that the \emph{sub-gradient} of a convex function $h$ at a point $x\in \mathbb{R}^n$ is the following set valued function  
$$dh(x):=\{q\in \mathbb{R}^n \,:\, h(y) \geq h(x)+\langle q, y-x\rangle , \; \forall y\in \mathbb{R}^n \}.$$

The Legendre transform of $h_P$ takes values $0$ on $P$ and $+\infty$ on $\Bbb R^n \setminus P$. Hence $dh_P(\mathbb{R})\subset P$. Since $dh_P(0)=P$ it then follows that  $dh_P(\mathbb{R}^n)=P$.
For a smooth convex function $h:\Bbb R^n \to \Bbb R$ the \emph{real Monge-Amp\`ere measure}  of $h$ is defined as 
$$
\textup{MA}_{\Bbb R}(h):= \frac{n!}{2^n} \det \left(\frac{\partial^2h}{\partial x_j \partial x_k} \right)dx_1 \wedge \ldots \wedge d x_n.
$$
We chose the factor $\frac{n!}{2^n}$ to allow for a very attractive identity between the real and complex Monge-Amp\`ere operators below (see \eqref{eq: real_complex_MA_rel}). If $h$ is strictly convex and smooth then, by a change of variables, the formula above can be rewritten in the following form 
\begin{equation}\label{eq: Alexandrov_measure}
\MA_{\mathbb{R}} (h)(E)= \frac{n!}{2^n} \int_{dh(E)} d\,\Vol_n,
\end{equation}
where $E\subset \mathbb{R}^n$ is a Borel set, and $\textup{Vol}_n$ is the Euclidean measure of $\Bbb R^n$.  This allows to define  the \emph{real Monge-Amp\`ere measure} (in the sense of Alexandrov) for any convex function $h: \mathbb{R}^n \rightarrow \mathbb{R}$, with \eqref{eq: Alexandrov_measure} still in place (see \cite[Proposition 3.1]{RT69}). In particular, we have  
\begin{equation}\label{vol convex body}
\int_{\mathbb{R}^n} \MA_{\mathbb{R}} (h_P)=  \frac{n!}{2^n} \Vol_n(P).
\end{equation}

We define the Log map $L: (\mathbb{C}^*)^n\rightarrow \mathbb{R}_+^n$ by $L(z)=L(z_1, \dots , z_n)= (\log |z_1|, \dots, \log |z_n|).$ 
The \emph{logarithmic indicator function} of a convex body $P\subset \mathbb{R}^n$  is defined as 
\begin{equation*}
H_P(z) :=  h_P \circ L (z) = \sup_{p\in P} \log |z^p| := \sup_{p\in P} \log \left ( |z_1|^{p_1} ... |z_n|^{p_n} \right ).
\end{equation*}
In particular, $H_P$ is a plurisubharmonic function on $(\mathbb{C}^*)^n$. 

In analogy with the above, we consider the following class of plurisubharmonic functions (see \cite[page 10]{Ba17}, or \cite[Section 4]{Be09}):
$$\mathcal{L}_P((\mathbb{C}^*)^n):= \{\psi\in \psh((\mathbb{C}^*)^n) \;:\; \psi \leq H_P +C_\psi \; {\rm{on}}\; (\mathbb{C}^*)^n\}.$$
Using this terminology, it is elementary to see that 
\begin{equation}\label{eq: real_complex_CMA_op_id}
\mathcal{L}_P(\mathbb{R}^n)\circ L=\{\psi\in  \mathcal{L}_P((\mathbb{C}^*)^n)\,:\, \psi\; \rm{is}\;  (S^1)^n {- \rm invariant}\}.
\end{equation}
It is well known (\cite[Lemma 2.2 and Section 2.2]{BB13}) that, given $(S^1)^n$-invariant $\psi\in  \mathcal{L}_P((\mathbb{C}^*)^n)$, $\psi=h\circ L$ with $h\in \mathcal{L}_P(\mathbb{R}^n)$, the real and complex Monge-Amp\`ere measures satisfy
\begin{equation}\label{eq: real_complex_MA_rel}
L_\star (i\partial \bar{\partial}\psi)^n = \MA_{\mathbb{R}}(h).
\end{equation}
It then follows that $\int_E \MA_{\mathbb{R}} (h)= \int_{L^{-1}(E)} (i \partial \bar{\partial} \psi)^n$
for any Borel set $E\subset \mathbb{R}^n$. In particular, the above combined with \eqref{vol convex body} tells us that
\begin{equation}\label{vol convex body 2}
\int_{(\mathbb{C}^*)^n}(i \partial \bar{\partial} H_P)^n = \frac{n!}{2^n}\Vol_n(P).
\end{equation}

In what follows we only consider convex bodies $P \subset \mathbb{R}_+^n$, unless otherwise specified. All corresponding results for general convex bodies $P$ can be easily deduced by making a translation, however we choose to avoid the cumbersome notation that comes with the treatment of arbitrary $P$.

Given a convex body $P \subset \mathbb{R}_+^n$, let $r>0$ be big enough such that $P \subset r \Sigma$, where $\Sigma$ is the unit simplex in $\mathbb{R}^n$. 
Noting that $(\mathbb{C^*})^n \subset \mathbb{C}^n \cup H_\infty=\mathbb{CP}^n$,  recall that $\omega_{FS}|_{(\Bbb {C^*})^n} =\frac{i}{2}\partial\bar \partial \log(1 + \| z\|^2)$ and the $(S^1)^n$-action of $(\Bbb C^*)^n$ extends to an action on $\Bbb{CP}^n$. For $z\in (\mathbb{C}^*)^n$ we define
\begin{equation}
	\label{eq: model_function_P}
	\phi_P(z):= H_{P}(z) -\frac{r}{2} \log (1+\|z\|^2).
\end{equation}
The choice of $r$ ensures that   $\phi_P$ is bounded from above on $(\mathbb{C}^*)^n$. Since $\Bbb{CP}^n\setminus (\mathbb{C}^*)^n$ is pluripolar, $\phi_P$ can be extended as a $r \omega_{FS}$-psh function on $\Bbb {CP}^n$. 
Moreover, by \eqref{vol convex body 2} and the fact that the non-pluripolar product does not charge pluripolar sets, we have
\begin{equation}\label{volume compact-real}
\int_{\mathbb{CP}^n} \left(r \omega_{FS}+ i \partial \bar{\partial} \phi_{P}\right)^n =\int_{(\mathbb{C}^*)^n}(i \partial \bar{\partial} H_{P})^n= \int_{\mathbb{R}^n} \MA_{\mathbb{R}}( h_{P})=  \frac{n!}{2^n}\Vol_n (P).
\end{equation}
Let $P_1,\ldots,P_n \subset \Bbb R_+^n$ such that  $P_j \subset r\Sigma$ for some $r>0$. By the above we also have:
\begin{equation}\label{mixed volume compact-real}
\int_{\mathbb{CP}^n} \left(r \omega_{FS}+ i \partial \bar{\partial} \phi_{P_1}\right)\wedge \cdots \wedge \left(r \omega_{FS}+ i \partial \bar{\partial} \phi_{P_n}\right) =\int_{(\mathbb{C}^*)^n}(i \partial \bar{\partial} H_{P_1}) \wedge \cdots \wedge (i \partial \bar{\partial} H_{P_n}).
\end{equation}

The construction  in \eqref{eq: model_function_P} gives rise to the following bijection
 $$ 
 \tau_P:\mathcal{L}_P((\mathbb{C}^*)^n) \xrightarrow{\sim} \{\varphi\in \psh(\mathbb{CP}^n, r\omega_{FS}) \,:\, \varphi  \preceq\phi_{P} \}.
 $$
Restricting to $(S^1)^n$-invariant elements we get another bijection, again denoted by $\tau_P$:
\begin{equation}\label{eq: convex_S_1_isom}
\tau_P:\mathcal{L}_P(\mathbb{R}^n) \xrightarrow{\sim} \{\varphi\in \psh(\mathbb{CP}^n, r\omega_{FS})\,:\, \varphi  \preceq\phi_{P} \; \rm{and}\; \varphi\mid_{(\mathbb{C}^*)^n} \; \textup{is}\; {(S^1)^n} \textup{-invariant}\}.
\end{equation}
Since the $(S^1)^n$-action of $(\Bbb C^*)^n$ extends to an action on $\Bbb{CP}^n$, with an abuse of terminology, in what follows we will say that $\varphi\in \psh(\mathbb{CP}^n, r\omega_{FS})$ is $(S^1)^n$-invariant meaning that $\varphi$ is invariant under the extended action on the whole $\Bbb{CP}^n$.

\subsection{Real Monge-Amp\`ere equations}

We fix momentarily a convex body $P \subset \Bbb R_+^n$ and $r>0$ such that $P \subset r\Sigma$.
Following the terminology of \cite[Section 2.3.1]{BB13}, we say that $h\in \mathcal{L}_P(\mathbb{R}^n)$ has $P$-relative full mass, i.e., 
$$h \in \mathcal E_P(\mathbb{R}^n) \subset  \mathcal{L}_P(\mathbb{R}^n),$$ 
if $\int_{\Bbb R^n} \MA_{\mathbb{R}}(h)=\int_{\mathbb{R}^n } \MA_{\mathbb{R}}(h_P)$ . 

 We start our analysis with a simple consequence of \eqref{volume compact-real} that gives a clear relation between the classes $\mathcal{E}_P(\mathbb{R}^n)$ and $\mathcal{E}(\mathbb{C}\mathbb{P}^n, r\omega_{FS}, \phi_P)$:
\begin{prop}\label{masses real-compact} The following hold:
\begin{itemize}
\item[(i)] if $h\in \mathcal{E}_P(\mathbb{R}^n)$, then $\tau_P(h)\in \mathcal{E}(\mathbb{C}\mathbb{P}^n, r\omega_{FS}, \phi_P)$;
\item[(ii)] if $\varphi\in\mathcal{E}(\mathbb{C}\mathbb{P}^n, r\omega_{FS}, \phi_P) $ and $\varphi$ is $(S^1)^n$-invariant, then $\tau_P^{-1}(\varphi) \in \mathcal{E}_P(\mathbb{R}^n)$.
\end{itemize}
\end{prop}

Next we point out a fundamental result linking $(S^1)^n$-invariant model type singularities and convex bodies:
\begin{theorem}\label{thm: model_sing_convex_body} Fix $r>0$. Then the following hold:
\begin{itemize}
	\item [(i)]  Given a convex body $P \subset r \Sigma$, the potential $\phi_P \in \textup{PSH}(\Bbb{CP}^n,r\omega_{FS})$ from \eqref{eq: model_function_P} has model type singularity and $\int_X (r\omega_{FS}+ i\ddbar \phi_P)^n=\frac{n!}{2^n}\Vol_n(P)>0$.
	\item [(ii)] Given $\phi \in \textup{PSH}(\Bbb{CP}^n, r\omega_{FS})$ with model type singularity  that is $(S^1)^n$-invariant and $\int_{\Bbb{CP}^n} (r\omega_{FS} + i\ddbar \phi_P)^n>0$, there exists a convex body $P \in r\Sigma$ such that $[\phi]=[\phi_P]$.
\end{itemize}
\end{theorem}
\begin{proof}
First we argue (i). By \cite[Proposition 2.8]{BB13}, the set 
$${\mathcal{L}_P(\mathbb{R}^n)}_0:= \left\{h \in\mathcal{L}_P(\mathbb{R}^n)\, :\, \sup_{\mathbb{R}^n} (h-h_P)=0 \right\}$$ is compact.
Moreover, we have that
$\sup_{\mathbb{R}^n} (h-h_P) = \sup_{\mathbb{CP}^n} \big( \tau_P(h \circ L)-\phi_P\big).$
Consequently the following set is also compact:
$$\Big\{ \varphi\in \psh(\mathbb{CP}^n,r\omega_{FS})\,:\, \varphi  \preceq\phi_{P}, \;  \varphi \; {\rm{is}}\; (S^1)^n\textup{-invariant 
 and }  \sup_{\mathbb{CP}^n} (\varphi-\phi_P)=0  \Big\}.$$
Since $\phi_P$ is $(S^1)^n$-invariant, so is $P_{r\omega_{FS}}[\phi_P]$, and exactly the same argument as the one in  Lemma \ref{lem: compactness and model type} ensures that $\phi_P$ has model type singularity. From \eqref{volume compact-real} we also have that 
$$
\int_{\mathbb{CP}^n} (r\omega_{FS} + i\ddbar \phi_P)^n= \frac{n!}{2^n}\Vol_n(P)>0.
$$

Now we argue (ii). Using the construction of \eqref{eq: model_function_P}, to $\phi$  we associate a convex function $h: \Bbb R^n \to \Bbb R$ and \eqref{eq: real_complex_CMA_op_id} such that $h(x) \leq \frac{r}{2} \log (1+e^{2x_1}+...+e^{2x_n}) + C$ for some $C>0$, and
$$\int_{\mathbb{R}^n} \MA_{\mathbb{R}} (h) =\int_{(\mathbb{C^*})^n} \Big(i\ddbar \big(\phi(z) + \frac{r}{2}\log(1 + \|z\|^2)\big)\Big)^n = \int_{\mathbb{C}\mathbb{P}^n} (r\omega_{FS} + i\ddbar \phi)^n.$$
Then the closure of the set $dh(\mathbb{R}^n)$ is a convex body in $\mathbb{R}_+^n$ which will be denoted by $P$. Since $h \leq \frac{r}{2} \log (1+e^{2x_1}+...+e^{2x_n}) + C$ it follows that $P\subset r\Sigma$. 

By convexity of $P$ the Euclidean measure of $\partial P$ is zero, implying that 
$$\int_{\mathbb{R}^n} \MA_{\mathbb{R}^n}(h_P)=\Vol_n(P)=\Vol(dh(\Bbb R^n))=\int_{\mathbb{R}^n} \MA_{\mathbb{R}^n}(h)>0.$$
By comparing the support of the Legendre transforms, it follows that $h \leq h_P + C$. Together with the above, this gives $h\in \mathcal{E}_P(h_P)$, further implying that $\phi \in \mathcal{E}(\mathbb{CP}^n,r\omega_{FS}, \phi_P)$. 
Now \cite[Theorem 1.3 (iii)]{DDL2} implies that $P[\phi_P]=P[\phi]$. It follows from the first part of the theorem that $\phi_P$ has model type singularity.  Since $[\phi]$ is also a model type singularity, by definition we obtain that $[\phi] = [P[\phi]]=[P[\phi_P]]=[\phi_P]$, finishing the proof.
\end{proof}

With the duality of the above two results in hand, we can provide the real  Monge-Amp\`ere analog of Theorem \ref{thm: non pluripolar existence}, which recovers a result of  Berman-Berndtsson \cite[Theorem 2.19]{BB13}, obtained using completely different variational techniques: 
\begin{theorem}\label{resolution real MA}
Let $P$ be a convex body in $\mathbb{R}^n$ and let $\mu$ be a positive Borel measure on $\mathbb{R}^n$ such that  $\mu (\mathbb{R}^n)= \int_{\mathbb{R}^n} \MA_{\mathbb{R}}(h_P)$. Then there exists $h\in \mathcal{E}_P(\mathbb{R}^n)$, unique up to an additive constant, such that
\begin{equation}\label{real MA}
\MA_{\mathbb{R}}(h)=\mu.
\end{equation}
\end{theorem}
For conditions on  the measure $\mu$ which guarantee that the solutions $h$ have the same singularity type as $h_P$, we refer to the next remark.

\begin{proof} We can assume that $P \subset r \Sigma$ for some $r>0$. This can always be obtained after a translation and big enough $r>0$. Such a translation will only change the desired solution $h$ by a linear term.

Let $\tilde{\mu}$ be the non-pluripolar measure on $\mathbb{CP}^n $ that is $(S^1)^n$-invariant with $L_* \tilde{\mu}=\mu$. Given this choice, it is clear that 
$$\tilde \mu(\mathbb{CP}^n) = \int_{\mathbb{CP}^n} \left(r\omega_{FS}+ i \partial \bar{\partial} \phi_{P}\right)^n=\int_{\mathbb{R}^n} \MA_{\mathbb{R}}(h_P).$$ 
Now the result follows after an application of Theorem \ref{thm: non pluripolar existence} to $\tilde \mu$ and the model singularity type $[\phi_P]$ (Theorem \ref{thm: model_sing_convex_body}(i)). Indeed, uniqueness guarantees that a solution $u \in \mathcal E(\Bbb{CP}^n,r\omega_{FS},\phi_P)$ to the equation
\begin{equation}\label{eq: CMAE_eq_from_P}
(r\omega_{FS} + i\ddbar u)^n = \tilde \mu
\end{equation}
is $(S^1)^n$-invariant, since so is the data. Proposition \ref{masses real-compact} then immediately gives that $u = \tau_P(h)$, for some $h \in \mathcal E_P (\Bbb R^n)$ that solves  \eqref{real MA} and is unique up to a constant.
 \end{proof}
 
 \begin{remark}\label{rem:BB13 2.23}
 	In \cite[Remark 2.23]{BB13} Berman and Berndtsson ask whether boundedness of the solution to \eqref{real MA} follows from the following integrability condition: 
 \begin{equation}\label{eq: integrability q 1}
\int_{\mathbb{R}^n} |g- h_P|^{n+\delta} d\mu <+\infty,\qquad \forall g \in \mathcal{E}_P(\mathbb{R}^n).
\end{equation}
The answer is yes, and we summarize our reasoning. Via compactness, condition \eqref{eq: integrability q 1} translates to relative pluripotential theory (as explained above) in the following form: there exists a constant $C_0>0$ such that 
 \begin{equation*}
\int_{\mathbb{\mathbb{C}\mathbb{P}}^n} |\varphi- \phi_P|^{n+\delta} d \tilde{\mu} \leq  C_0
\end{equation*}
for all $\varphi  \in \mathcal{E}(\mathbb{C}\mathbb{P}^n, r\omega_{FS}, \phi_P)$ which are $(S^1)^n$-invariant and $\sup_{\mathbb{C} \mathbb{P}^n} (\varphi-\phi_P)=0$. The above estimate then gives a volume-capacity comparison for $(S^1)^n$-invariant Borel sets $E$: 
$$
\tilde{\mu}(E) \leq C_1 \left [ \capacity_{\phi_P}(E) \right ]^{1+\varepsilon}.
$$
 Since both the solution $u \in \mathcal E(\Bbb{CP}^n,r \omega_{FS}, \phi_P)$ (to the equation \eqref{eq: CMAE_eq_from_P}) and the model potential $\phi_P$ are $(S^1)^n$-invariant,  it follows that the sublevel sets $\{u <\phi_P -t\}, t>0$ are also $(S^1)^n$-invariant. 
With this in hand, the proof of Theorem \ref{thm: uniform estimate} carries over (with $a=0$, $d \tilde \mu$ in place of $f \omega^n$, and $\phi_P$ in place of $\chi$) giving the global boundedness of $u-\phi_P$.	
 \end{remark}

\subsection{The Brunn-Minkowski inequality}

The Brunn-Minkowski inequality plays a central role in many branches of analysis and geometry,  especially in the theory of convex bodies. We refer to the beautiful survey of R. Gardner \cite{Ga02} for an extensive account on the subject.

Given two convex bodies $P_1,P_2 \subset \Bbb R^n$ we denote their Minkowski sum by
$$P_1+P_2:=\{p_1+p_2\,:\, p_1\in P_1, p_2\in P_2\}.$$

Minkowski showed that if $P_1, \ldots, P_k, k\leq n$ are convex bodies in $\mathbb{R}^n$ and $t_1, \ldots, t_k\geq 0$, the volume $\Vol_n (t_1P_1+\ldots+t_kP_k)$ is a polynomial of degree $n$ in the variables $t_1, \ldots, t_k$. In the special case $k=n$, the coefficient of  $t_1 t_2\ldots t_n$ in this polynomial is $n!\MV(P_1, \ldots, P_n)$, where $\MV(P_1,...,P_n)$ is the \emph{mixed volume} of $P_1,\ldots,P_n$. Here we choose the factor $n!$ to ensure that $\MV(P,...,P)=\Vol_n(P)$.
\medskip

Lastly, we point out that Theorem B is the complex analog of the celebrated Brunn-Minkowski inequality  (and its variants):

\begin{theorem}\label{thm: BM}
Let $P_1, \cdots, P_n$ be convex bodies in $\mathbb{R}^n$. Then \\
\noindent (i) $\MV(P_1, \cdots, P_n) \geq  \Vol_n(P_1)^{1/n}\ldots \Vol_n(P_n)^{1/n}.$\\
\noindent (ii)$\left(\Vol_n(tP_1+(1-t)P_2)\right)^{\frac{1}{n}} \geq t \Vol_n(P_1)^{\frac{1}{n}}+ (1-t)\Vol_n(P_2)^{\frac{1}{n}}.$
\end{theorem}

\begin{proof}
To start, after making a translation an choosing a big enough $r>0$ we can assume that $P_j\subset r\Sigma$, for all $j$. Comparing the support functions we deduce that 
$$H_{t_1P_1+\ldots+ t_nP_n}= t_1 H_{P_1}+\ldots + t_n H_{P_n}.$$ 
It thus follows from \eqref{volume compact-real} that
\begin{flalign}
\int_{(\mathbb{C}^*)^n} \left(i \partial \bar{\partial} (t_1 H_{P_1}+\ldots + t_n H_{P_n})\right)^n &= \int_{(\mathbb{C}^*)^n} \left(i \partial \bar{\partial} H_{t_1 P_1 + \ldots + t_n P_n} \right)^n \nonumber \\
&= \frac{n!}{2^n} \textup{Vol}_n(t_1P_1+\ldots+ t_n P_n).
\end{flalign}
Since the first and last expressions are homogeneous polynomials of degree $n$ in the variables $t_1, \ldots, t_n$, by comparing coefficients and using \eqref{mixed volume compact-real} we arrive at (c.f. \cite[Proposition 2.4]{Ba17}):
\begin{flalign*}\label{mixed vol}
\int_{\mathbb{CP}^n} \left(r\omega_{FS}+ i \partial \bar{\partial} \phi_{P_1}\right)\wedge \ldots \wedge \left(r\omega_{FS}+ i \partial \bar{\partial} \phi_{P_n}\right)  &= \int_{(\mathbb{C}^*)^n}i \partial \bar{\partial} H_{P_1}\wedge \ldots \wedge i \partial \bar{\partial} H_{P_n}\\  
&= \frac{n!}{2^n} \,\MV(P_1, \ldots, P_n).
\end{flalign*} 
Putting this together with Theorem \ref{thm: log concavity} and \eqref{volume compact-real},  the inequality of (i) immediately follows.

To argue (ii) one simply expands $\Vol_n(tP_1+(1-t)P_2)$ using multilinearity. Then an application of the inequality of (i) for each resulting term yields the desired conclusion.
\end{proof}

\footnotesize
\let\OLDthebibliography\thebibliography 
\renewcommand\thebibliography[1]{
  \OLDthebibliography{#1}
  \setlength{\parskip}{1pt}
  \setlength{\itemsep}{1pt}
}

\normalsize
 \noindent{\sc University of Maryland}\\
{\tt tdarvas@math.umd.edu}\vspace{0.05in}

\noindent{\sc IHES}\\
{\tt dinezza@ihes.fr} \vspace{0.05in}

\noindent {\sc Universit\'e Paris-Sud}\\
{\tt hoang-chinh.lu@math.u-psud.fr}
\end{document}